\documentclass[10pt]{elsarticle}
\usepackage{amsmath,amssymb,amsthm,amscd}

\def \N {\mathbb{N}}

\def \A {\mathcal{A}}
\def \B {\mathcal{B}}
\def \S {\mathcal{S}}
\def \ol {\overline}
\def \sgp {\operatorname{sgp}}

\usepackage{tikz}
\usetikzlibrary{calc}

\usetikzlibrary{automata,shapes.geometric,shapes.misc,shapes.symbols}

\tikzset{
  dominopicture/.style={
    baseline=0mm,
    dominoinode/.style={
      font=\small,
      inner sep=0mm,
      outer sep=0mm,
      minimum height=2.5mm,
      minimum width=2.5mm,
      execute at begin node=\strut
    },
  },
  smalldominopicture/.style={
    baseline=-.75mm,
    dominoinode/.style={
      font=\small,
      inner sep=0mm,
      outer sep=0mm,
      minimum height=2mm,
      minimum width=1mm,
    },
  },
}

\newcommand\domino[3][dominopicture]{
  \begin{tikzpicture}[#1]
    \node[anchor=base,dominoinode] (lnode) at (0,0) {$#2$};
    \node[anchor=base west,dominoinode] (rnode) at ($ (lnode.base
east) + (1mm,0) $) {$#3$};
    \coordinate (nw) at ($ (current bounding box.north west) + (-1mm,.15mm) $);
    \coordinate (se) at ($ (current bounding box.south east) + (1mm,-.15mm) $);
    \coordinate (mid) at ($ (lnode.east)!.5!(rnode.west) $);
    \draw ($ (mid |- nw) $)--($ (mid |- se) $);
    \draw (nw) rectangle (se);
  \end{tikzpicture}
}
\newcommand\tablet[3][dominopicture]{
  \begin{tikzpicture}[#1]
    \node[anchor=base,dominoinode] (lnode) at (0,0) {$#2$};
    \node[anchor=base west,dominoinode] (rnode) at ($ (lnode.base
east) + (1mm,0) $) {$#3$};
    \coordinate (nw) at ($ (current bounding box.north west) + (.5mm,.15mm) $);
    \coordinate (se) at ($ (current bounding box.south east) +
(-.5mm,-.15mm) $);
    \coordinate (mid) at ($ (lnode.east)!.5!(rnode.west) $);
    \draw ($ (mid |- nw) $)--($ (mid |- se) $);
    \draw (nw) -- (nw -| se);
    \draw (se) -- (se -| nw);
    \coordinate (emid) at ($ (nw -| se)!.5!(se) $);
    \coordinate (e) at ($ (emid)!1!90:(se) $);
    \coordinate (wmid) at ($ (se -| nw)!.5!(nw) $);
    \coordinate (w) at ($ (wmid)!1!90:(nw) $);
    \draw[out=0,in=90] (nw -| se) to (e);
    \draw[out=-90,in=0] (e) to (se);
    \draw[out=180,in=-90] (se -| nw) to (w);
    \draw[out=90,in=180] (w) to (nw);
  \end{tikzpicture}
}

\newcommand\domi{\domino[smalldominopicture]{}{}\kern.3mm}
\newcommand\tabl{\tablet[smalldominopicture]{}{}\kern.3mm}

\newcommand{\fsa}[1]{\mathcal{#1}}
\newcommand{\defterm}[1]{\textit{#1}}

\DeclareMathOperator{\End}{End}

\let\emptyword=\varepsilon
\newcommand{\rel}[1]{\mathcal{#1}}

\newtheorem{theorem}{Theorem}
\newtheorem{proposition}[theorem]{Proposition}
\newtheorem{lemma}[theorem]{Lemma}

\newtheorem{example}[theorem]{Example}
\newtheorem{question}[theorem]{Question}

\theoremstyle{definition}
\newtheorem{definition}[theorem]{Definition}

\usepackage{hyperref}

\begin{document}

\begin{frontmatter}

\title{Automaton semigroups: new constructions results and
examples of non-automaton semigroups}

\author[a]{Tara Brough\corref{cor1}\fnref{fn1}}
\ead{trbrough@fc.ul.pt}

\author[b]{Alan J.\ Cain\corref{cor2}\fnref{fn2}}
\ead{a.cain@fct.unl.pt}

\date{}

\cortext[cor1]{Principal Corresponding Author}
\cortext[cor2]{Corresponding Author}
\address[a]{Centro de Matem\'{a}tica Computacional e Estoc\'{a}stica,\\
Departamento de Matem\'{a}tica, Faculdade de Ci\^{e}ncias, Universidade de Lisboa,\\ 1749-016 Lisboa, Portugal }
\address[b]{Centro de Matem\'{a}tica e Aplica\c{c}\~{o}es, Faculdade de Ci\^{e}ncias e Tecnologia \\
Universidade Nova de Lisboa, 2829--516 Caparica, Portugal}

\fntext[fn1]{Much of the research leading to this paper was undertaken during a visit by the first author to the Centro
  de Matem\'{a}tica e Aplica\c{c}\~{o}es, Universidade Nova de Lisboa. We thank the centre and university for their
  hospitality. This visit was partially supported by the \textsc{FCT} exploratory project
  \textsc{IF}/01622/2013/\textsc{CP}1161/\textsc{CT}0001 (attached to the second author's fellowship). The first
  author was also supported by the FCT project CEMAT-CI\^{E}NCIAS (UID/Multi/04621/2013).}
\fntext[fn2]{The second author was
  supported by an Investigador \textsc{FCT} fellowship (\textsc{IF}/01622/\allowbreak 2013/\textsc{CP}1161/{\sc
    CT}0001). This work was partially supported by the Funda\c{c}\~{a}o para a Ci\^{e}ncia e a Tecnologia (Portuguese
  Foundation for Science and Technology) through the project {\sc UID}/{\sc MAT}/00297/2013 (Centro de Matem\'{a}tica e
  Aplica\c{c}\~{o}es).}

\begin{abstract}
  This paper studies the class of automaton semigroups from two perspectives: closure under constructions, and examples
  of semigroups that are not automaton semigroups. We prove that (semigroup) free products of finite semigroups always
  arise as automaton semigroups, and that the class of automaton monoids is closed under forming wreath products with
  finite monoids.  We also consider closure under certain kinds of Rees matrix constructions, strong semilattices, and
  small extensions. Finally, we prove that no subsemigroup of $(\mathbb{N},+)$ arises as an automaton
  semigroup. (Previously, $(\mathbb{N},+)$ itself was the unique example of a semigroup having the `general' properties
  of automaton semigroup (such as residual finiteness, solvable word problem, etc.) but that was known not to arise as
  an automaton semigroup.)
\end{abstract}

\begin{keyword}
Automaton semigroup \sep Free product \sep Wreath product \sep Constructions \sep Small extensions
\MSC[2010]   20M35 \sep 68Q45 \sep 68Q70
\end{keyword}

\end{frontmatter}

\section{Introduction}

Automaton semigroups (that is, semigroups of endomorphisms of rooted trees generated by the actions of Mealy automata)
emerged as a generalisation of automaton groups, which arose from the construction of groups having `exotic' properties,
such as the finitely generated infinite torsion group found by Grigorchuk \cite{grigorchuk_burnside}, and later
proven to have intermediate growth, again by Grigorchuk \cite{grigorchuk_degrees}. The topic of automaton groups has
since developed into a substantial theory; see, for example, Nekrashevych's monograph \cite{nekrashevych_self} or one of
the surveys by the school led by Bartholdi, Grigorchuk, Nekrashevych, and \v{S}uni\'{c}
\cite{bartholdi_fractal,bartholdi_branch,grigorchuk_self}.

After the foundational work of Grigorchuk, Nekrashevych \& Sushchanskii~\cite[esp.~Sec.~4 \&~Subsec.~7.2]{grigorchuk_automata},
the theory of automaton semigroups has grown into an active research topic. Broadly speaking, there have been two foci
of research. First, the study of decision problems: what can be effectively decided about the semigroup generated by
a given automaton? For example, the finiteness and torsion problems are now known to be undecidable for general
automaton semigroups \cite{gillibert_finiteness}, but particular special cases are decidable
\cite{klimann_implementing,klimann_finiteness,maltcev_cayley}. Second, the study of the class of automaton semigroups:
which semigroups arise and do not arise as automaton semigroups? Two particular aspects of this question are whether the
class of automaton semigroups is closed under various semigroup constructions, and giving examples of semigroups that do
not arise as automaton semigroups. This paper is concerned with both of these aspects.

For some constructions, such as direct products and adjoining a zero or identity, it is straightforward to prove that
the class is closed; see \cite[Section~5]{c_1auto}. For many other natural constructions, the question of closure
remains open. For example, whether automaton semigroups are closed under free products is an open question. (This is
related to the problem of showing that all free groups arise as automaton groups; the recent positive answer to this
question was the culmination of the work a series of authors; see \cite{steinberg_automata} and the references therein.)
The free product of automaton semigroups is, however, at least very close to being an automaton semigroup:
in a previous paper, we showed that $(S\star T)^1$ is always an automaton semigroup if $S$ and $T$ are
\cite[Theorem~3]{bc_freeprodauto}.

Since closure under free products seemed difficult to settle, the second author asked whether free products of finite
semigroups always arise as automaton semigroups \cite[Open problem~5.8(1)]{c_1auto}. In \cite[Conjecture~5]{bc_freeprodauto},
we conjectured that the answer was `no', and suggested a potential counterexample. However, in
this paper we prove that the answer is `yes': free products of finite semigroups always arise as automaton semigroups
(Theorem~\ref{thm:finfree}). This parallels the result that (group) free products of finite groups arise as automaton groups
\cite{gupta_freeprod}.  More generally, we show in Theorem~\ref{thm:freeprodidpt} that the free product of
automaton semigroups each containing an idempotent is always an automaton semigroup.

In our previous paper, we also considered whether a wreath product $S \wr T$, where $S$ is an automaton monoid and $T$ is
a finite monoid, was necessarily an automaton monoid. We managed to prove that such a wreath product arises as a
\emph{submonoid} of an automaton monoid. In this paper, we obtain a complete answer: all such wreath products arise as
automaton monoids (Theorem \ref{thm:wreath}).

We consider whether a Rees matrix semigroup over an automaton semigroup is also an automaton semigroup. We do not have a
complete answer, but we prove that this holds under certain restrictions (Proposition~\ref{prop:rees}).
This is a step towards classifying completely
simple automaton semigroups \cite[Open problem~5.8(3)]{c_1auto}.

We prove that a certain kind of strong semilattice of automaton semigroups is itself an automaton semigroup
(Proposition~\ref{prop:semilattice}). This result is then applied when we turn to the question of whether a small
extension of an automaton semigroup is necessarily an automaton semigroup. (Recall that if $S$ is a semigroup and $T$ is
a subsemigroup of $S$ with $S \setminus T$ finite, then $S$ is a small extension of $T$ and $T$ is a large subsemigroup
of $S$.) Many finiteness properties are known to be preserved on passing to large subsemigroups and small extensions;
see the survey \cite{cm_finreessurvey}. It is already known that a large subsemigroup of an automaton semigroup is not
necessarily an automaton semigroup, for $(\N\cup\{0\},+)$ is an automaton semigroup but $(\N,+)$ is not; see
\cite[Section~5]{c_1auto}. (This example also shows that adding an identity to a non-automaton semigroup can give an
automaton monoid; thus the classes of automaton semigroups and monoids seem to be very different.) We
do not have a complete answer to the question of closure under small extensions, but we prove some special cases in
Section~\ref{sec:smallext}. The importance of these results is that \emph{if} the class of automaton semigroups is not
closed under forming small extensions, then we have eliminated several standard constructions as potential sources of
counterexamples.

In all of the automaton constructions in this paper, we use alphabets of symbols consisting at least
partially of tuples of symbols from the automata for the `base' semigroups of the construction.
This seems to be quite a powerful approach, as it allows the automaton to have access to a lot of
information at each transition.

Finally, we present new examples of semigroups that do not arise as automaton semigroups. This is an important advance,
because a major difficulty in studying the class of automaton semigroup is that if a semigroup has the properties that
automaton semigroups have generally, such as residual finiteness \cite[Proposition~3.2]{c_1auto}, solvable word problem,
etc., then there are no general techniques for proving it is not an automaton semigroup. In the pre-existing literature,
there is a unique example of a semigroup that has these `general' automaton semigroup properties but that is known
\emph{not} to arise as an automaton semigroup: namely, the free semigroup of rank $1$ (or, if one prefers, $(\N,+)$)
\cite[Proposition 4.3]{c_1auto}. We prove that no subsemigroup of this semigroup arises as an automaton semigroup, and
indeed that no subsemigroup of this semigroup with a zero adjoined arises as an automaton semigroup
(Theorem~\ref{thm:nosubsemigroupofn}). Although our proof is specialised, and thus still leaves open the problem of
finding a general technique for proving that a semigroup is not an automaton semigroup, we at least now have a
countable, rather than singleton, class of non-automaton semigroups that satisfy the usual general properties of
automaton semigroups.

\section{Preliminaries}

In this section we briefly recall the necessary definitions and concepts required in the rest of the paper. For a fuller
introduction to automaton semigroups, see the discussion and examples in \cite[Sect.~2]{c_1auto}.

An \defterm{automaton} $\fsa{A}$ is formally a triple $(Q,B,\delta)$,
where $Q$ is a finite set of \defterm{states}, $B$ is a finite
alphabet of \defterm{symbols}, and $\delta$ is a transformation of the
set $Q \times B$. The automaton $\fsa{A}$ is normally viewed as a
directed labelled graph with vertex set $Q$ and an edge from $q$ to
$r$ labelled by $x \mid y$ when $(q,x)\delta = (r,y)$:
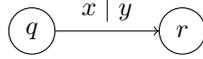
\[
\begin{tikzpicture}[every state/.style={inner sep=1mm,minimum size=1.75em}]
\node[state] (q) at (0,0) {$q$};
\node[state] (r) at (2,0) {$r$};
\draw[->] (q) edge node[anchor=south] {$x\mid y$} (r);
\end{tikzpicture}
\]
The interpretation of this is that if the automaton $\fsa{A}$ is in
state $q$ and reads symbol $x$, then it changes to the state $r$ and
outputs the symbol $y$. Thus, starting in some state $q_0$, the
automaton can read a sequence of symbols
$\alpha_1\alpha_2\ldots\alpha_n$ and output a sequence
$\beta_1\beta_2\ldots\beta_n$, where $(q_{i-1},\alpha_i)\delta =
(q_i,\beta_i)$ for all $i = 1,\ldots,n$.

Such automata are more usually known in computer science as deterministic real-time (synchronous) transducers, or Mealy
machines. In the field of automaton semigroups and groups, they are simply called `automata' and this paper retains this
terminology.

Each state $q \in Q$ acts on $B^*$, the set of finite sequences of
elements of $B$. The action of $q \in Q$ on $B^*$ is defined as
follows: $\alpha \cdot q$ (the result of $q$ acting on $\alpha$) is
defined to be the sequence the automaton outputs when it starts in the
state $q$ and reads the sequence $\alpha$. That is, if $\alpha =
\alpha_1\alpha_2\ldots\alpha_n$ (where $\alpha_i \in B$), then $\alpha
\cdot q$ is the sequence $\beta_1\beta_2\ldots\beta_n$ (where $\beta_i
\in B$), where $(q_{i-1},\alpha_i)\delta = (q_i,\beta_i)$ for all $i =
1,\ldots,n$, with $q_0 = q$.

The set $B^*$ can be identified with an ordered regular tree of degree
$|B|$. The vertices of this tree are labelled by the elements of
$B^*$. The root vertex is labelled with the empty word $\emptyword$,
and a vertex labelled $\alpha$ (where $\alpha \in B^*$) has $|B|$
children whose labels are $\alpha\beta$ for each $\beta \in B$. It is
convenient not to distinguish between a vertex and its label, and thus
one normally refers to `the vertex $\alpha$' rather than `the vertex
labelled by $\alpha$'. (Figure~\ref{fig:rootedtree} illustrates the
tree corresponding to $\{0,1\}^*$.)

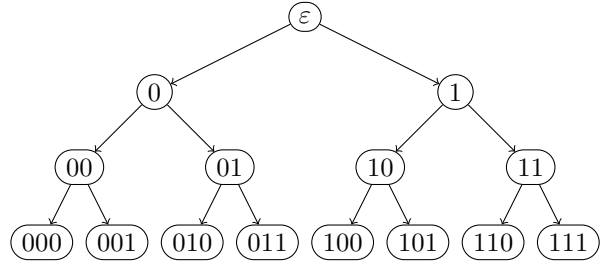
\begin{figure}[tb]
\centering
\begin{tikzpicture}[every node/.style={rounded rectangle,draw},every child/.style={->},
level distance=10mm,level 1/.style={sibling distance=40mm},level 2/.style={sibling distance=20mm},level 3/.style={sibling distance=10mm}]
  \node (root) {$\emptyword$}
    child { node (0) {$0$}
      child { node (00) {$00$}
        child { node (000) {$000$} }
        child { node (001) {$001$} } }
      child { node (01) {$01$}
        child { node (010) {$010$} }
        child { node (011) {$011$} } } }
    child { node (1) {$1$}
      child { node (10) {$10$}
        child { node (100) {$100$} }
        child { node (101) {$101$} } }
      child { node (11) {$11$}
        child { node (110) {$110$} }
        child { node (111) {$111$} } } };
\end{tikzpicture}
\caption{The set $\{0,1\}^*$ viewed as a rooted binary tree.}
\label{fig:rootedtree}
\end{figure}

The action of a state $q$ on $B^*$ can thus be viewed as a
transformation of the corresponding tree, sending the vertex $w$ to
the vertex $w \cdot q$. Notice that, by the definition of the action
of $q$, if $\alpha\alpha' \cdot q = \beta\beta'$ (where $\alpha,\beta
\in B^*$ and $\alpha',\beta' \in B$), then $\alpha \cdot q =
\beta$. In terms of the transformation on the tree, this says that if
one vertex ($\alpha$) is the parent of another ($\alpha\alpha'$), then
their images under the action by $q$ are also parent ($\beta$) and
child ($\beta\beta'$) vertices. More concisely, the action of $q$ on
the tree preserves adjacency and is thus an endomorphism of the
tree. Furthermore, the action's preservation of lengths of sequences
becomes a preservation of levels in the tree.

The actions of states extends naturally to actions of words: $w =
w_1\cdots w_n$ (where $w_i \in Q$) acts on $\alpha \in B^*$ by
\[
(\cdots((\alpha \cdot w_1)\cdot w_2) \cdots w_{n-1})\cdot w_n.
\]

So there is a natural homomorphism $\phi : Q^+ \to \End B^*$, where
$\End B^*$ denotes the endomorphism semigroup of the tree $B^*$. The
image of $\phi$ in $\End B^*$, which is necessarily a semigroup, is
denoted $\Sigma(\fsa{A})$.

A semigroup $S$ is called an \defterm{automaton semigroup} if there
exists an automaton $\fsa{A}$ such that $S \simeq \Sigma(\fsa{A})$.

It is often more convenient to reason about the action of a state or
word on a single sequence of infinite length than on sequences of some
arbitrary fixed length. The set of infinite sequences over $B$ is
denoted $B^\omega$. The infinite sequence consisting of countably many
repetitions of the finite word $\alpha \in B^*$ is denoted
$\alpha^\omega$. For synchronous automata, the action on infinite
sequences determines the action on finite sequences and \emph{vice
  versa}.

The following lemma summarises the conditions under which two words
$w$ and $w'$ in $Q^+$ represent the same element of the automaton
semigroup. The results follow immediately from the definitions, but are
so fundamental that they deserve explicit statement:

\begin{lemma}
\label{lem:eq}
Let $w,w' \in Q^+$. Then the following are equivalent:
\begin{enumerate}
\item $w$ and $w'$ represent the same element of $\Sigma(\fsa{A})$;
\item $w\phi = w'\phi$;
\item $\alpha\cdot w = \alpha\cdot w'$ for each $\alpha \in B^*$;
\item $w$ and $w'$ have the same actions on $B^n$ for every $n \in
  \N_0$;
\item $w$ and $w'$ have the same actions on $B^\omega$.
\end{enumerate}
\end{lemma}

Generally, there is no need to make a notational distinction between
$w$ and $w\phi$. Thus $w$ denotes both an element of $Q^+$ and the
image of this word in $\Sigma(\fsa{A})$. In particular, one writes `$w
= w'$ in $\Sigma(\fsa{A})$' instead of the strictly correct `$w\phi =
w'\phi$'. With this convention, notice that $Q$ generates
$\Sigma(\fsa{A})$.

Some further notation is required for the rest of the paper:
For $w \in Q^+$, define $\tau_w : B \to B$ by $b \mapsto b \cdot
w$. For $b \in B$, define $\pi_b : Q \to Q$ by $q \mapsto r$ if
$(q,b)\delta = (r,x)$ for some $x \in B$ (in fact, $x = b\tau_q$). So
$q\pi_b$ is the state to which the edge from $q$ labelled by
$b\mid\cdot$ leads. Thus $(q,b)\delta = (q\pi_b,b\tau_q)$.

Further, let $w \in Q^+$. For any $\alpha \in B^*$, there is a unique $w|_\alpha \in \End B^*$ such that
$\alpha\beta\cdot w = (\alpha\cdot w)(\beta\cdot w|_\alpha)$; see \cite{nekrashevych_self} for details. Notice that if $w,w' \in Q^+$ are equal in $\Sigma(\fsa{A})$, then $w|_\alpha = w'|_\alpha$ for all
$\alpha \in B^*$.

We now recall the notion of \defterm{wreath recursions}. The endomorphism semigroup of $B^*$ decomposes as a recursive
wreath product:
\[
\End B^* = \End B^* \wr \rel{T}_B,
\]
where $\rel{T}_B$ is the transformation semigroup of the set $B$. That is,
\[
\End B^* = \big(\underbrace{\End B^* \times \ldots \times \End B^*}_{\text{$|B|$ times}}\big) \rtimes \rel{T}_B
\]
where $\rel{T}_B$ acts from the right on the coordinates of elements
of the direct product of $|B|$ copies of $\End B^*$.

If $p \in \End B^*$ with $p = (x_0,x_1,\ldots,x_{|B|-1})\tau$, then
$\tau$ describes the action of $p$ on $B$ and each $x_i$ is an element
of $\End B^*$ whose action on $B^*$ mirrors the action of $p$ on the
subtree $b_iB^*$. Alternatively: to act on $B^*$ by $p$, act on each
subtree $b_iB^*$ by $x_i$, and then act on the collection of the
resulting subtrees according to $\tau$.

Thus, if $p,q \in \End B^*$ with $p = (x_0,x_1,\ldots,x_{|B|-1})\tau$ and $q = (y_0,y_1,\ldots,y_{|B|-1})\rho$, where
$\tau,\rho \in \rel{T}_B$ and $x_i,y_j \in \End B^*$, then
\begin{equation}
pq = (x_0y_{0\tau},x_1y_{1\tau},\ldots,x_{|B|-1}y_{(|B|-1)\tau})\tau\rho. \label{eq:wreath}
\end{equation}

If $p \in Q$, then $\tau = \tau_p$ and $x_i = p\pi_i$. That is,
\[
p = (p\pi_1,p\pi_2,\ldots,p\pi_{|B|-1})\tau_p.
\]
This description of the action of $p$ is called a \defterm{wreath
  recursion}. Its primary use is to calculate, by means of the
multiplication given in \eqref{eq:wreath}, the action of a word $w \in
Q^+$ on $B^*$.

\section{Free products}

The \emph{free product} of two semigroups $S = \sgp \langle X_1 \mid R_1\rangle$ and
$T = \sgp \langle X_2 \mid R_2\rangle$, denoted $S\star T$, is the semigroup
with presentation $\sgp \langle X_1\cup X_2 \mid R_1\cup R_2\rangle$.

In \cite[Conjecture~5]{bc_freeprodauto}, the present authors conjectured that
there exist finite semigroups $S$ and $T$ such that $S\star T$ is not an
automaton semigroup.  We begin by showing that this is not the case.

\begin{theorem}\label{thm:finfree}
Let $S$ and $T$ be finite semigroups.  Then $S\star T$ is an automaton semigroup.
\end{theorem}

\begin{proof}
Let $e$ and $f$ be distinguished idempotents of $S$ and $T$ respectively.
Let $\A = (Q,B,\delta)$ with $Q = Q_1\cup Q_2$, where $Q_1$ is a copy of $S$
and $Q_2$ is a copy of $T$. Define an alphabet
\[ B = \left\{ \domino{}{}, \domino{s}{}, \domino{s}{t}, \domino{s}{t}^\circ,
\tablet{}{}, \tablet{t}{}, \tablet{t}{s}, \tablet{t}{s}^\circ \mid s\in S, t\in T\right\} \]
and let $\delta$ be the transformation of $Q\times B$ given by the following transition table.
\[
{\renewcommand\arraystretch{1.3}
\begin{array}{|l|c|c|}
\hline
 & s & t\\
\hline
\domino{}{}   & (f,\domino{s}{})            &   (f,\domino{}{})\\
\domino{a}{}  & (f,\domino{as}{})          &    (f,\domino{a}{t})\\
\domino{a}{b} & (s,\domino{a}{b}^\circ)    &    (f,\domino{a}{bt})\\
\domino{a}{b}^\circ & (s,\domino{a}{b}^\circ)  &   (t,\domino{a}{b}^\circ)\\
\tablet{}{}  & (e,\tablet{}{})           &    (e,\tablet{t}{})\\
\tablet{b}{} & (e,\tablet{b}{s})          &    (e,\tablet{bt}{})\\
\tablet{b}{a} & (e,\tablet{b}{as}) &    (t,\tablet{b}{a}^\circ)\\
\tablet{b}{a}^\circ & (s,\tablet{b}{a}^\circ) &   (t,\tablet{b}{a}^\circ)\\[1mm]
\hline
\end{array}}
\]
for $s\in Q_1$, $t\in Q_2$, $a\in S$ and $b\in T$,
where $as$ and $bt$ denote elements of $S$ and $T$ respectively rather than two-letter words.

We will refer to $\domi$-symbols and $\tabl$-symbols, meaning all symbols having those shapes.
Actions on strings of $\domi$-symbols will help us distinguish words beginning with an element
of $S$, while actions on strings of $\tabl$-symbols help us distinguish words begining with
and element of $T$.
We call a symbol \emph{full} if it has entries in both boxes, \emph{open} if not, and \emph{marked}
if it has the $^\circ$ superscript.
Notice that all states `ignore' marked symbols: that is, if $x^\circ$ is a marked symbol,
then $(q,x)\delta = (q,x)$ for all $q\in Q$.

We begin by showing that $\A$ defines actions of $S$ and $T$, and hence of $S\star T$, on $B^*$.
Firstly, we consider only the states $e$ and $f$, whose actions are illustrated in Figure~\ref{fig:eandf}.
These states are of particular significance since, as can be seen from the definition of $\delta$,
all transitions lead either back to the state they started from or to one of $e$ or $f$.

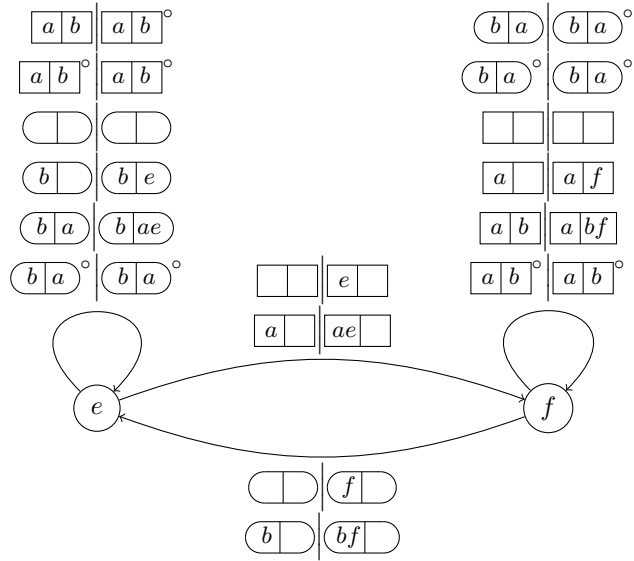
\begin{figure}[t]
\centering
\begin{tikzpicture}[x=60mm,y=10mm,every state/.style={inner sep=1mm,minimum size=1.75em}]
%
\node[state] (e) at (0,0) {$e$};
\node[state] (f) at (1,0) {$f$};
\draw[->] (e) edge[bend left=20] (f);
\draw[->] (f) edge[bend left=20] (e);
\draw[->] (e) edge[loop above,min distance=20mm,out=135,in=45] (e);
\draw[->] (f) edge[loop above,min distance=20mm,out=135,in=45] (f);
\node[anchor=south,align=center] at (.5,.6) {$\domino{}{}\Big|\domino{e}{}$\\$\domino{a}{}\Big|\domino{ae}{}$};
\node[anchor=north,align=center] at (.5,-.6) {$\tablet{}{}\Big|\tablet{f}{}$\\$\tablet{b}{}\Big|\tablet{bf}{}$};
\node[anchor=south,align=center] at (0,1.3) {$\phantom{{}^\circ}\domino{a}{b}\Big|\domino{a}{b}^\circ$\\
$\domino{a}{b}^\circ\Big|\domino{a}{b}^\circ$\\$\tablet{}{}\Big|\tablet{}{}$\\
$\tablet{b}{}\Big|\tablet{b}{e}$\\$\tablet{b}{a}\Big|\tablet{b}{ae}$\\$\tablet{b}{a}^\circ\Big|\tablet{b}{a}^\circ$};
\node[anchor=south,align=center] at (1,1.3) {$\phantom{{}^\circ}\tablet{b}{a}\Big|\tablet{b}{a}^\circ$\\
$\tablet{b}{a}^\circ\Big|\tablet{b}{a}^\circ$\\
$\domino{}{}\Big|\domino{}{}$\\$\domino{a}{}\Big|\domino{a}{f}$\\
$\domino{a}{b}\Big|\domino{a}{bf}$\\$\domino{a}{b}^\circ\Big|\domino{a}{b}^\circ$};
\end{tikzpicture}
\caption{Actions of $e$ and $f$.}
\label{fig:eandf}
\end{figure}

The state $e$ has no effect on marked symbols or the symbol $\tabl$, marks any full $\domi$-symbol,
and multiplies the second entry of non-empty $\tabl$-symbols by $e$ (or inserts $e$ if
it is blank) -- returning to state $e$ on all of these actions --
while on open $\domi$-symbols it multiplies the first entry
by $e$ (or inserts $e$ if it is blank) and moves to $f$.  The action of $f$ can be described by
switching the roles of $e$ and $f$ and of $\domi$-symbols and $\tabl$-symbols in the preceding sentence.

Let $B_e = B\setminus \{\domino{}{},\domino{a}{}: a \in S\}$ and
$B_f = B\setminus \{\tablet{}{},\tablet{b}{}: b \in T\}$.
The above discussion and the fact that $e$ and $f$ are idempotents in $S$ and $T$ respectively
implies that $e$ and $f$ act as idempotents on $B_e^*$ and $B_f^*$ respectively.
To see how $e$ acts on a string in $B^*$, the important symbols to take note of are the
$\domi$-symbol with empty second component, then the next $\tabl$-symbol with empty second
component, and so on alternatingly, since these are the symbols that will cause the automaton to change state.
So we write each string as a prefix of some alternating product of strings in
$B_e$ and $B_f$, distinguishing the important symbols, as follows (where $a_i$ or $b_i$ may denote empty space):
\[\alpha_1 \domino{a_1}{} \beta_1 \tablet{b_1}{} \ldots \alpha_i \domino{a_i}{} \beta_i \tablet{b_i}{}\ldots \in B^\omega,\]
where ${\alpha_i\in B_e^*}$ and ${\beta_i\in B_f^*}$.  Then
\begin{align*}
&\alpha_1 \domino{a_1}{} \beta_1 \tablet{b_1}{} \ldots \alpha_i \domino{a_i}{} \beta_i \tablet{b_i}{}\ldots \cdot e\\
=& (\alpha_1\cdot e) \, \domino{a_1 e}{} \, (\beta_1\cdot f) \, \tablet{b_1 f}{} \ldots
(\alpha_i\cdot e) \, \domino{a_i e}{} \, (\beta_i\cdot f) \, \tablet{b_i f}{}\ldots.
\end{align*}
(If, for example, $a_i$ denotes an empty space, then $a_i e = e$.)
Acting on the resulting string by $e$ again has the result of replacing each $e$ and $f$ by $e^2$ and $f^2$
respectively,
but since we already know that $e$ acts idempotently on $S$ and $B_e^*$, while $f$ acts idempotently on
$T$ and $B_f^*$, this makes no change.  Hence $e^2 = e$ in $\Sigma(\A)$.  Similarly, to show that $f^2=f$
in $\Sigma(\A)$, we would express strings in $B^*$ in the form $\beta_1 \tablet{b_1}{} \alpha_1 \domino{a_1}{} \ldots$.

We can now describe the action of $Q_1$ on $B^*$.
Each state in $Q_1$ recurses to itself on marked symbols (which it leaves unchanged)
and on full $\domi$-symbols (which it marks); to $e$ on unmarked $\tabl$-symbols;
and to $f$ on open $\domi$-symbols.
Let $C$ be the set of marked symbols and full $\domi$-symbols in $B$, and for
$\alpha\in C^*$, let $\alpha^\circ$ denote the word obtained from $\alpha$ by marking all unmarked symbols.
We can express any string in $B^*$ in the form $\alpha \beta \gamma$, where
$\alpha\in C^*$, $\beta\in B\setminus C$, $\gamma\in B^*$.  Let $s_1,\ldots, s_n\in S$.
Since the type ($\domi$ or $\tabl$) of the symbol $\beta$ is not changed by the action of any state,
and also $\beta\cdot s\notin C$ for any $s\in S$, we have for some $\epsilon\in \{e,f\}$
\begin{align*}
\alpha\beta\gamma \cdot s_1\ldots s_k &= [\alpha^\circ (\beta\cdot s_1) (\gamma\cdot \epsilon)] \cdot s_2\ldots s_n\\
&= \ldots = \alpha^\circ (\beta\cdot s_1\ldots s_n) (\gamma\cdot \epsilon^n)\\
&= \alpha^\circ (\beta\cdot s_1\ldots s_n) (\gamma\cdot \epsilon),
\end{align*}
Thus the action of $\langle Q_1\rangle$ on $B^*$ depends only on its action on $B\setminus C$.
(Note that the idempotency of $e$ and $f$ is critical in establishing this.)
Let $w = s_1\ldots s_n\in Q_1^+$ and let $s_w$ be the element of $S$ represented by $w$. Then
\begin{align*}
\domino{}{}\cdot w &= \domino{s_w}{} \\
\domino{a}{}\cdot w &= \domino{as_w}{} \\
\tablet{}{}\cdot w &= \tablet{}{}\\
\tablet{b}{}\cdot w &= \tablet{b}{s_w}\\
\tablet{b}{a}\cdot w &= \tablet{b}{as_w}.
\end{align*}
This shows that the action of $w$ on $B^*$ depends only on $s_w\in S$,
so that $\langle Q_1\rangle$ must be isomorphic to some quotient of $S$.

By symmetry of the construction, we also find that $\langle Q_2 \rangle$ is isomorphic
to some quotient of $T$, and so $\A$ defines an action of $S\star T$ on $B^*$.

It remains to prove that this action is faithful. We have already seen that the actions
of words in $Q_1^+$ and $Q_2^+$ depend only on the elements of $S$ and $T$ respectively
that they represent, so it suffices to consider the action of reduced words.
The idea of this automaton is that the action on the string $\domi^\omega$ can be used to recover
any reduced word in $S\star T$ starting with an element of $S$, while $\tabl^\omega$ is used to recover
reduced words starting with elements of $T$.  Given a word $s_1 t_1 \ldots s_k t_k$ with $s_i\in S$, $t_i\in T$, we have
\begin{align*}
\domino{}{}^\omega \cdot s_1 t_1 \ldots s_k t_k &= \domino{s_1}{} \, \domino{}{}^\omega \cdot t_1 s_2 \ldots s_k t_k\\
&= \domino{s_1}{t_1} \, \domino{}{}^\omega \cdot s_2 t_2 \ldots s_k t_k\\
&= \domino{s_1}{t_1}^\circ \, \domino{s_2}{t_2} \, \domino{}{}^\omega \cdot s_3 t_3 \ldots s_k t_k\\
&= \domino{s_1}{t_1}^\circ \, \domino{s_2}{t_2}^\circ \ldots \domino{s_{k-1}}{t_{k-1}}^\circ \domino{s_k}{t_k}.
\end{align*}
If the final $t_k$ is not present, the resulting final symbol will instead be $\domino{s_k}{}$.
Thus we can read off any reduced word $w$ starting with an element from $S$ from the string
$\domi^\omega \cdot w$.  Similarly, if $w$ is a reduced word starting with an element from $T$,
we can read it off from the string $\tabl^\omega \cdot w$.
Moreover,
\begin{align*}
\domino{}{}^\omega\cdot t_1 s_1 \ldots t_k s_k &=
\domino{s_1}{t_2}^\circ \ldots \domino{s_{k-1}}{t_k}^\circ \, \domino{s_k}{},\\
\tablet{}{}^\omega \cdot s_1 t_1 \ldots s_k t_k &=
\tablet{t_1}{s_1}^\circ \ldots \tablet{t_{k-1}}{s_k}^\circ \, \tablet{t_k}{}.
\end{align*}
Hence pairs of distinct elements of $S\star T$ can be distinguished by their actions on
one of $\domi^\omega$ or $\tabl^\omega$, and so $\Sigma(\A)\cong S\star T$.
\end{proof}

We can in fact considerably generalise the construction in the preceding proof: the
important point is the existence of idempotents in the factor semigroups.
The following theorem generalises \cite[Theorem~2]{bc_freeprodauto},
which says that the free product of automaton semigroups $S$ and $T$
is an automaton semigroup if $S$ and $T$ each contain a left identity.

\begin{theorem}\label{thm:freeprodidpt}
Let $S$ and $T$ be automaton semigroups each containing at least one idempotent.
Then $S\star T$ is an automaton semigroup.
\end{theorem}

This is immediate from the following more technical result:

\begin{theorem}\label{thm:freeprodfull}
Let $S$ and $T$ be automaton semigroups and suppose that there exist
$e\in S$, $f\in T$ such that
\[
  \bigl(u =_S u' \lor  u =_T u'\bigr) \implies \bigl(e^{|u|} = e^{|u'|} \land f^{|u|} = f^{|u'|}\bigr).
\]
Then $S\star T$ is an automaton semigroup.
\end{theorem}

\begin{proof}
Let $e$ and $f$ be distinguished elements of $S$ and $T$ respectively
satisfying the hypothesis of the theorem.
(For example, $e$ and $f$ might be idempotents.)
Let $\A_1 = (Q_1,A,\delta_1)$ and $\A_2 = (Q_2,B,\delta_2)$ be automata
for $S$ and $T$ respectively.  We may assume that $e\in Q_1$ and $f\in Q_2$.
Let $X = \{ \domino{a}{b}, \tablet{b}{a} \mid a\in A, b\in B \}$
and $Y = \{\$,\#\}$.
We shall call the symbols in $X$ \emph{dominoes} and the symbols in $Y$ \emph{gates}.
We construct an automaton $\A = (Q,C,\delta)$ with $Q = Q_1\cup Q_2$,
\[ C = \{x, x^S, x^T, x^\circ,
y, \hat{y}, \ol{y}
\mid x\in X, y\in Y\} \]
and $\delta$ the transformation of $Q\times C$ defined as follows.
For $s\in Q_1$, $t\in Q_2$, $a\in A$, $b\in B$
suppose that $(s,a)\delta_1 = (s_0,a_0)$ and $(t,b)\delta_2 = (t_0,b_0)$.  Then
the action of $Q$ on $\domi$-symbols and $\$$-gates is given by
\[
{\renewcommand\arraystretch{1.3}
\begin{array}{|c|c|c|}
\hline
 & s & t\\
\hline\
\domino{a}{b}  & (s_0, \domino{a_0}{b}^S)    &      (t,\domino{a}{b})\\
\domino{a}{b}^S & (s_0, \domino{a_0}{b}^S)   &     (t_0, \domino{a}{b_0}^T)\\
\domino{a}{b}^T & (s, \domino{a}{b}^\circ)  &    (t_0, \domino{a}{b_0}^T)\\
\domino{a}{b}^\circ & (s,\domino{a}{b}^\circ)  &  (t,\domino{a}{b}^\circ)\\
\ol{\$} & (f,\hat{\$})                      &   (f,\ol{\$})\\
\hat{\$} & (f,\hat{\$})                     &   (f,\$)\\
\$ & (s,\$^\circ)                           &   (f,\$)\\
\$^\circ & (s,\$^\circ)                     &   (t,\$^\circ)\\[1mm]
\hline
\end{array}}
\]
The action of $Q$ on the remainder of $C$ ($\tabl$-symbols and $\#$-gates) is given by
replacing each $\domino{i}{j}$ in the above table by $\tablet{j}{i}$ and swapping the
corresponding symbols in the tuples $(S,s,s_0,f,\$)$ and $(T,t,t_0,e,\#)$.

For $x\in X$ and $y\in Y$, we call $x$ \emph{unmarked}, $x^S$ \emph{$S$-marked},
$x^T$ \emph{$T$-marked}, $x^\circ$ and $y^\circ$ \emph{circled},
$y$ \emph{open}, $\hat{y}$ \emph{half-open} and $\ol{y}$ \emph{closed}.

This construction is inspired by the construction in Theorem~\ref{thm:finfree}.
Since single symbols are no longer sufficient for distinguishing elements of $S$
and of $T$, we instead use strings of several $\domi$- or $\tabl$-symbols,
separated by either $\$$-gates or $\#$-gates.

We first describe the action of a word in $Q_1Q^+$ on a string consisting only
of unmarked $\domi$-symbols and closed $\$$-gates.  Let $w = u_1v_1\ldots u_kv_k$
with $u_i\in Q_1^+$ and $v_i\in Q_2^+$ and let
$\alpha = \alpha_1 \ol{\$} \alpha_2 \ldots \ol{\$} \alpha_k$ with each $\alpha_i$
consisting only of unmarked $\domi$-symbols (note that $\alpha_i$ may be empty).
Then $u_1$ acts on $\alpha$ by acting on the first entries of $\alpha_1$ just
as in $\A_1$, $S$-marking the resulting $\domi$-symbols, and half-opens the first
$\$$-gate, after which the automaton moves to state $f$ and thus leaves the rest of the string unchanged.
Next, $v_1$ acts on $\alpha\cdot u_1$ by acting on the second entries of
$\alpha_1\cdot u_1$ just as in $\A_2$, $T$-marking the resulting $\domi$-symbols,
opening the first $\$$-gate, and leaving the rest of the string unchanged.
Now $\alpha\cdot u_1v_1$ begins with a string of $T$-marked $\domi$-symbols,
followed by an open $\$$-gate.  The first state in $u_2$ circles the
initial string of $\domi$-symbols and the first $\$$-gate, all the while not
changing state.
By induction, we have
\[
\alpha\cdot w = (\alpha_1\cdot u_1v_1)^\circ \$^\circ (\alpha_2\cdot u_2v_2)^\circ \$^\circ \ldots \$^\circ (\alpha_k\cdot u_kv_k).
\]
(All symbols up until the last $\$$-gate are circled.)
Thus if $w' = u_1' v_1' \ldots u_k' v_k'$ is another word with $u_i'\in Q_1^+$
and $v_i'\in Q_2^+$ and some $u_i'\neq_S u_i$, we can distinguish $w$ and $w'$
as follows.  Let $\gamma = \ol{\$}^{i-1}\beta$, where $\beta$ is some string
of $\domi$-symbols such that reading off the first entries of $\beta$ gives a
word which $u_i$ and $u_i'$ act differently on.  Then
\begin{align*}
\gamma\cdot w = (\$^\circ)^{i-1} (\beta\cdot u_iv_i)^{(\circ)} &\neq (\$^\circ)^{i-1} (\beta\cdot u'_iv'_i)^{(\circ)} = \gamma\cdot w',
\end{align*}
where the parentheses around the superscript $\circ$ indicate that the $\circ$ is present if and only if $i\neq k$.
If instead some $v_i\neq_T v_i'$, then the same idea using second entries
instead of first entries for $\beta$ works.
Words in $Q_1Q^+$ representing elements of $S\star T$ of different reduced lengths
can be distinguished by their actions on the string $\ol{\$}^\omega$.
(The \emph{reduced length} of an element $s\in S\star T$ is the length of
an alternating product of elements of $S$ and $T$ representing $s$.)

Words in $Q_2Q^+$ have an analogous action to the one described above on
strings consisting only of unmarked $\tabl$-symbols and closed $\#$-gates,
and can thus be distinguished similarly.
Two words not starting with symbols from the same $Q_i$ can be distinguished
their actions on either $\ol{\$}^\omega$ or $\ol{\#}^\omega$ (usually both).

It remains to show that $\A$ defines an action of $S\star T$ on $C^*$.
For this, it suffices to show that the action of $Q_1^+$ gives an action of
$S$, since it will follow by symmetry of the construction that the action of
$Q_2^+$ gives an action of $T$.

Let $w\in Q_1^+$ and $\alpha\in C^*$.
For clarity, we explain the action of $w$ on $\alpha$ by a series of observations.
\begin{enumerate}
\item
When acting on $C^*$ by $Q_1^+$ , certain symbols are `uninteresting', in the sense
that the same thing happens to them when acted on by any word in $Q_1^+$, and they
also do not affect what happens to the rest of the string containing them.
All circled symbols are uninteresting, as are $\$, \hat{\$}, \#$,
unmarked $\tabl$-symbols and $T$-marked $\domi$-symbols.
We may thus assume that $\alpha$ contains none of these symbols;
that is, that \[\alpha\in \{\ol{\$}, \ol{\#}, \hat{\#}, \domino{a}{b}, \domino{a}{b}^S,
\tablet{b}{a}^T, \tablet{b}{a}^S \mid a\in A, b\in B\}^*.\]

\item
Furthermore, under actions of $Q_1^+$, the following pairs of symbols are essentially
the same: $(\ol{\#}, \hat{\#})$, $(\domino{a}{b}, \domino{a}{b}^S)$ and
$(\tablet{b}{a}^T, \tablet{b}{a}^S)$.  This is because the action of $Q_1$ on the
two symbols in each pair is identical.  In each case, the output from both symbols
is a symbol of the second type, and the symbols of the first type do not occur in
$C^*\cdot Q_1^+$.  We may thus assume that
\[\alpha\in \{\ol{\$}, \hat{\#}, \domino{a}{b}^S, \tablet{b}{a}^S \mid a\in A, b\in B\}^*.\]

\item
For $\beta$ consisting only of $S$-marked dominoes, then let $\beta_A\in A^*$ be the word
obtained from $\beta$ by reading off the symbols from $A$ in each domino (which will
be the first entry for $\domi$-symbols and the second entry for $\tabl$-symbols).
Then $(\beta\cdot w)_A = \beta_A\cdot w$, where the action of $w$ is in $\A$ on the
left-hand side and in $\A_1$ on the right-hand side.
If we define $\beta_B$ similarly, then $(\beta\cdot w)_B = \beta_B$.
Hence $Q_1^+$ defines an action of $S$ on $S$-marked dominoes.\\
Moreover, note that for $v\in Q_2^+$ we have $(\beta\cdot v)_A = \beta_A$
and $(\beta\cdot v)_B = \beta_B\cdot v$.

\item
In general, $\alpha$ can be assumed to be a prefix of some
$\gamma = \alpha_1 y_1 \alpha_2 y_2 \ldots \alpha_k y_k$, where each $\alpha_i$
is a string of $S$-marked dominoes and each $y_i$ is a gate.
If $w$ has length $n$, then
\[ \gamma\cdot w = (\alpha_1\cdot w) y_1 (\alpha_2\cdot g_2^n) y_2 \ldots (\alpha_k\cdot g_k^n),\]
where $g_i$ is $e$ if $y_{i-1} = \hat{\#}$ and $f$ if $y_{i-1} = \ol{\$}$.
By (iii) and the hypothesis on $e$ and $f$, the string $\alpha_1\cdot w$ and
each $\alpha_i\cdot g_i^n$ depend only on the element of $S$ represented by $w$.
Hence we have $\alpha\cdot w = \alpha\cdot w'$ whenever $w=_S w'$, for $w, w'\in Q_1^+$ and
$\alpha\in C^*$.
\end{enumerate}

Thus $\A$ defines a faithful action of $S\star T$ on $C^*$ and so $S\star T$ is
an automaton semigroup.
\end{proof}

Aside from $S$ and $T$ containing idempotents, another way to satisfy the hypothesis of Theorem~\ref{thm:freeprodfull}
is for $S$ and $T$ to be \emph{homogeneous}, meaning that any two words representing the same element have the same
length.  (In this case $e$ and $f$ can be taken to be arbitrary elements of $S$ and $T$ respectively.)  Free semigroups
and free commutative semigroups of rank at least $2$ are automaton semigroups that have this property.  A more important
example is the plactic monoid (see, for example, \cite[Ch.~5]{lothaire_algebraic}), which Picantin has recently shown to
be an automaton semigroup \cite{picantin_plactic}.

The question of whether the class of automaton semigroups is closed under
taking free products remains open.
It is even possible that the condition in Theorem~\ref{thm:freeprodfull}
is necessary.  Unlike \cite[Theorem~2]{bc_freeprodauto},
Theorem~\ref{thm:freeprodfull} does account (by induction) for the free semigroups
and free monoids that can be constructed as free products of automaton
semigroups (i.e. free semigroups of rank at least $4$ and free monoids
of rank at least $2$).

\section{Wreath products}

The wreath product of two automaton semigroups is certainly not always an automaton semigroup, since it need not even be finitely generated.
One way to ensure that a wreath product $S\wr T$ is finitely generated is to require $S$ and $T$ to be monoids, with $T$ finite.
For monoids $S$ and $T$ with $T = \{t_1,\ldots,t_n\}$ finite, the \emph{wreath product} $S\wr T$ of $S$ with $T$ is a semidirect product $S^{|T|}\rtimes T$,
where $T$ acts on elements of $S^{|T|}$ by $(s_{t_1},s_{t_2},\ldots,s_{t_n})^t = (s_{t_1t}, s_{t_2t}, \ldots, s_{t_nt})$.

It turns out, contrary to \cite[Conjecture~6]{bc_freeprodauto}, that the wreath product of an automaton
monoid and a finite monoid is always an automaton monoid.

\begin{theorem}
\label{thm:wreath}
Let $S$ be an automaton monoid and $T$ a finite monoid.  Then $S\wr T$ is an automaton monoid.
\end{theorem}
\begin{proof}
First observe that $S^{|T|}$ (the direct product of $|T|$ copies of $S$) is an automaton semigroup by \cite[Proposition~5.5]{c_1auto},
and let $\A = (Q, A, \delta)$ be the standard automaton for $S^{|T|}$,
which has $Q = P^{|T|}$ (the Cartesian product of $|T|$ copies of $P$) for some generating set $P$ of $S$.

We construct an automaton $\B = (Q', C, \delta')$ with $\Sigma(\B) = S\wr T$.
Let $Q' = Q\times T$, $C = A\times B$, where $B$ is a copy of $T$,
and let $\delta': Q'\times C \rightarrow Q'\times C$ be given by:
\[ \left( (s,t), (a,b) \right) \mapsto \left( (s^b\pi_a, 1_T), (a\tau_{s^b}, bt) \right), \]
for $s\in Q$, $t\in T$, $a\in A$, $b\in B$, and $\pi_a$, $\tau_{s^b}$ are as in $\A$.
Note that after reading the first symbol in any $\alpha\in C^*$, the automaton $\B$ only utilises the
states of the form $(s,1_T)$, which act like $s$ on the first component of symbols in $C$ and leave the
second components unchanged. (See the example in Figure~\ref{fig:nwreathc2}.)

\begin{figure}[p]
\centering
\begin{tikzpicture}[x=40mm,y=40mm,every state/.style={rounded rectangle,inner sep=1mm,minimum size=1.75em}]
\node[state] (a) at (0,3.6) {$a$};
\node[state] (b) at (1,3.6) {$b$};
\begin{scope}[every node/.style={font=\footnotesize}]
  \draw[->] (a) edge[loop left] node {$*|*$} (a);
  \draw[->] (b) edge[loop right] node {$0|0$} (b);
  \draw[->] (b) edge node[anchor=south] {$1|0$} (a);
\end{scope}
\draw[lightgray,rounded corners,thick] (-.35,3.45) rectangle (1.35,3.75);
\node[state] (aa1) at (0,3) {$\bigl((a,a),1\bigr)$};
\node[state] (ab1) at (1,3) {$\bigl((a,b),1\bigr)$};
\node[state] (ba1) at (0,2) {$\bigl((b,a),1\bigr)$};
\node[state] (bb1) at (1,2) {$\bigl((b,b),1\bigr)$};
\node[state] (aac) at (0,1) {$\bigl((a,a),c\bigr)$};
\node[state] (abc) at (1,1) {$\bigl((a,b),c\bigr)$};
\node[state] (bac) at (0,0) {$\bigl((b,a),c\bigr)$};
\node[state] (bbc) at (1,0) {$\bigl((b,b),c\bigr)$};
\begin{scope}[every node/.style={font=\footnotesize}]
  \draw[->] (aa1) edge[loop above] node {$*|*$} (aa1);
  \draw[->] (ab1) edge[loop above] node[anchor=south,align=center] {$00e|00e$\\$10e|10e$} (ab1);
  \draw[->] (ba1) edge[loop left] node[align=center] {$00e|00e$\\$01e|01e$} (ab1);
  \draw[->] (bb1) edge[loop right] node[align=center,anchor=west] {$00e|00e$\\$00c|00c$} (bb1);
  \draw[->] (ab1) edge node[anchor=south,align=center] {$01e|00e$\\$10c|00c$\\
	$11e|10e$\\$11c|01c$} (aa1);
  \draw[->] (ab1) edge[bend left=5] node[pos=.3,anchor=north west,align=center,inner sep=.2mm] {$00c|00c$\\$01c|01c$} (ba1);
  \draw[->] (ba1) edge node[anchor=east,align=center] {$10e|00e$\\$01c|00c$\\
	$11e|01e$\\$11c|10c$} (aa1);
  \draw[->] (ba1) edge[bend left=5] node[pos=.3,anchor=south east,align=center,inner sep=.2mm] {$00c|00c$\\$10c|10c$} (ab1);
  \draw[line width=3pt,white] ($ (bb1) + (-.25,.25) $) -- ($ (aa1) + (.25,-.25) $);
  \draw[->] (bb1) edge node[rotate=-45,anchor=south,near start,inner sep=.3mm] {$11e|00e$} node[rotate=-45,anchor=north,near start,inner sep=.3mm] {$11c|00c$} (aa1);
  \draw[->] (bb1) edge node[near start,anchor=west,align=center] {$10e|00e$\\$10c|00c$} (ab1);
  \draw[->] (bb1) edge node[anchor=north,align=center,pos=.4] {$01e|00e$\\$01c|00c$} (ba1);
  \draw[rounded corners] (aac) -- node[anchor=south,align=center] {$\kappa \lambda e|\kappa \lambda c$} node[anchor=north,align=center] {$\kappa \lambda c|\kappa \lambda e$}
($ (aac) + (-35mm,0mm) $) -- ($ (aac) + (-35mm,10mm) $);
  \draw[rounded corners] (abc) -- node[pos=1,anchor=south,align=center,inner sep=.2mm] {$01e|00c$\;$10c|00e$\\$11e|10c$\;$11c|01e$} ($ (abc) + (-12.5mm,12.5mm) $) -- ($ (aac) + (-35mm,12.5mm) $) -- ($ (aac) + (-35mm,20mm) $);
  \draw[rounded corners] (abc) -- node[near start,anchor=north,align=center] {$00e|00c$\\$10e|10c$} ($ (abc) + (35mm,0) $) -- ($ (abc) + (35mm,10mm) $);
  \draw[rounded corners,->] (abc) -- node[near start,anchor=west,align=center] {$00c|00e$\\$01c|01e$} ($ (abc) + (0,.666) $) -- ($ (ba1) + (.333,-.333) $) -- (ba1);
  \draw[rounded corners,->] (bac) -- node[anchor=south west,align=center,inner sep=.2mm] {$00c|00e$\\$10c|10e$} ($ (bac) + (20mm,-20mm) $) -- ($ (bbc) + (35mm,-20mm) $) -- ($ (ab1) + (35mm,0mm) $) -- (ab1);
  \draw[rounded corners] (bac) -- node[very near start,anchor=south east,align=center,inner sep=.2mm]  {$01c|00e$\,$10e|00c$\\$11e|01c$\,$11c|01e$} ($ (bac) + (-35mm,0mm) $) -- ($ (aa1) + (-35mm,0mm) $) -- (aa1);
  \draw[line width=3pt,white] ($ (aac) + (.5,0) $) -- ($ (aac) + (0,.5) $);
  \draw[rounded corners,->] (bac) -- node[anchor=south east,align=center,inner sep=.2mm] {$00e|00c$\\$01e|01c$} ($ (bac) + (.5,.5) $) -- ($ (aac) + (.5,0) $) -- ($ (aac) + (0,.5) $) -- (ba1);
  \draw[line width=3pt,white] ($ (bbc) + (30mm,0mm) $) -- ($ (bb1) + (30mm,-30mm) $);
  \draw[rounded corners,->] (bbc) -- node[align=center,anchor=south] {$00e|00c$\\$00c|00e$} ($ (bbc) + (30mm,0mm) $) -- ($ (bb1) + (30mm,-30mm) $) -- ($ (bb1) + (10mm,-10mm) $) -- (bb1);
  \draw[line width=3pt,white] ($ (bbc) + (-15mm,-15mm) $) -- ($ (bac) + (-35mm,-15mm) $);
  \draw[rounded corners] (bbc) -- node[anchor=north west,align=center,near start,inner sep=.22mm] {$11e|00c$\\$11c|00e$} ($ (bbc) + (-15mm,-15mm) $) -- ($ (bac) + (-35mm,-15mm) $) -- ($ (bac) + (-35mm,10mm) $);
  \draw[rounded corners] (bbc) -- ($ (bbc) + (15mm,-15mm) $) -- node[anchor=south,align=center] {$10e|00c$\\$10c|00e$} ($ (bbc) + (35mm,-15mm) $) -- ($ (bbc) + (35mm,-5mm) $);
  \draw[rounded corners] (bbc) -- node[anchor=south west,align=center,inner sep=.2mm] {$01e|00c$\\$01c|00e$} ($ (bbc) + (-.5,.5) $) -- ($ (bbc) + (-.5,.75) $);
\end{scope}
\end{tikzpicture}
\caption{Automaton for the wreath product $\N_0
  \wr C_2$, where $C_2$ is the cyclic group of order $2$, with element set
  $\{e,c\}$.
  For reasons of space, symbols $\bigl((\kappa,\lambda),\mu\bigr)$
  are abbreviated $\kappa\lambda\mu$; thus $\bigl((1,0),c\bigr)$ is shown as
  $10c$. At the top of the diagram, in the grey box, is the original automaton for $\N_0$.}
\label{fig:nwreathc2}
\end{figure}
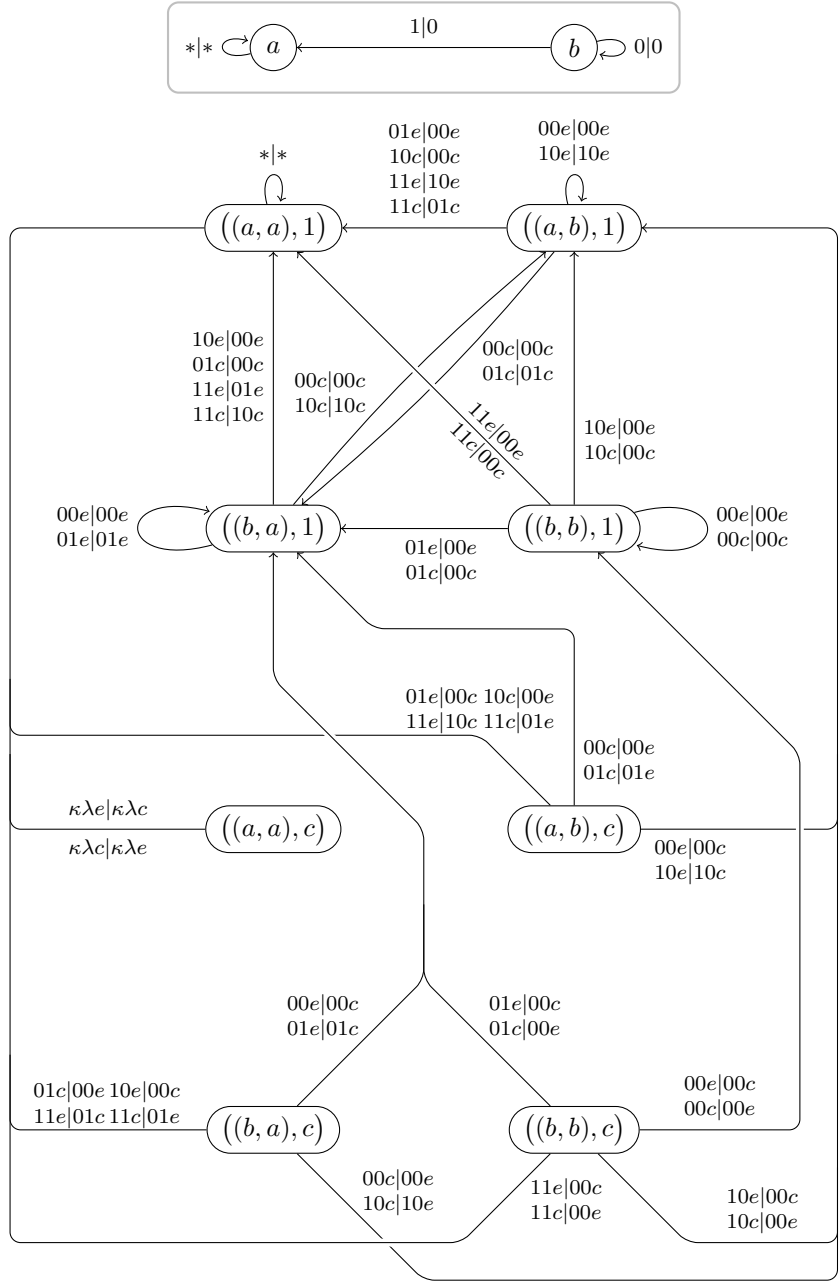

The `ideal' strings in $C^*$, which we use to distinguish elements of $S\wr T$, are those
from 
$D = (A\times T) (A\times \{1_T\})^*$.
For strings in $D$, we will simplify the notation by writing $(a_1,b_1)(a_2,1_T)\ldots (a_n,1_T)$
as $(a_1 a_2 \ldots a_n, b_1)$.  We have
\begin{align*}
(\alpha, 1_T) \cdot (s_1, t_1) (s_2, t_2) \ldots (s_m, t_m) &=
(\alpha\cdot s_1, t_1) \cdot (s_2, t_2) \ldots (s_m, t_m)\\
&= (\alpha\cdot s_1 s_2^{t_1}, t_1 t_2) \cdot (s_3, t_3) \ldots (s_m, t_m)\\
&= (\alpha\cdot s_1 s_2^{t_1} s_3^{t_1 t_2} \ldots s_m^{t_1 t_2\ldots t_m}, t_1 t_2 \ldots t_m).
\end{align*}
Since we already know that $\A$ gives a faithful action of $S^{|T|}$ on $A^*$, this shows that
$\B$ gives a faithful action of $S\wr T$ on $D^*$.  (Acting on $(\alpha, b)$ with $b\neq 1_T$
simply has the effect of premultiplying all the elements from $T$ by $b$.)
Hence any two words in $(Q')^*$ representing different elements of $S\wr T$ can be distinguished
by their actions on some word in $D\subset C^*$.

It is less obvious that $\B$ defines an action of $S\wr T$ on the remainder of $C^*$.
To see that it does, we view words in $C^*$ as a concatenation of words in $D$, and
then for $w\in (Q')^+$ we have:
\begin{multline*}
(\alpha_1, b_1) (\alpha_2, b_2) \ldots (\alpha_k, b_k)\cdot w\\
= [(\alpha_1, b_1)\cdot w] \; [(\alpha_2, b_2) \cdot (w_2,1_T)] \ldots [(\alpha_k, b_k)\cdot (w_k, 1_T)]
\end{multline*}
for some elements $w_j\in S$.  But $w_2$ is determined uniquely by $(\alpha_1,b_1)$ and the \emph{element}
of $S\wr T$ represented by $w$, and then each $w_j$ is determined recursively by all the $(\alpha_i, b_i)$
and $w_i$ for $i<j$ (setting $w_0 = w$), so that ultimately all the $w_j$ are determined uniquely by the string
$(\alpha_1, b_1)\ldots (\alpha_k, b_k)$ and the element represented by $w$.
Hence $\B$ does indeed define an action of $S\wr T$ on $C^*$, and so $S\wr T = \Sigma(\B)$.
\end{proof}

\section{Rees matrix semigroups}

Let us recall the definition of a Rees matrix semigroup. Let $M$ be a monoid, let $I$ and $\Lambda$ be abstract index
sets, and let $P\in \mathrm{Mat}_{\Lambda\times I}(M)$ (that is, $P$ is a $\Lambda \times I$ matrix with entries from
$M$). Denote the $(\lambda,i)$-th entry of $P$ by $p_{\lambda i}$. The \defterm{Rees matrix semigroup} over $M$ with
sandwich matrix $P$, denoted $\mathcal{M}[M;I,\Lambda;P]$, is the set $I \times M \times \Lambda$ with multiplication
defined by
\[
(i,x,\lambda)(j,y,\mu) = (i,xp_{\lambda j}y,\mu).
\]
The Rees matrix semigroup construction is particularly important because it arises in the classification of completely
simple semigroups; see \cite[Sect.~3.2--3.3]{howie_fundamentals} for background reading.

\begin{proposition}
\label{prop:rees}
Let $S$ be a Rees matrix semigroup $\mathcal{M}[M;I,\Lambda;P]$ with $I$ and $\Lambda$ finite and $M$ an automaton monoid
and $P\in \mathrm{Mat}_{\Lambda\times I}(M)$ a matrix containing the identity element of $M$ in some position.
If there exists an automaton for $M$ with state set $Q$ such that $1_M\in Q$ and $Q$ is closed under left-multiplication by
each non-zero entry of the matrix $P$, then $S$ is an automaton monoid.
\end{proposition}

[Note that if $P$ consists only of ones and zeros, then the hypothesis on $Q$ is always satisfied.]

\begin{proof}
Let $\A = (Q,A,\delta)$ be an automaton with $\Sigma(\A) = M$, such that
$Q$ satisfies the hypothesis of the theorem.
The fact that $1_M$ is in $Q$ and also appears in the matrix $P$ ensures that
$S = \mathcal{M}[M;I,\Lambda;P]$ is generated by the finite set $Q' = I\times Q\times \Lambda$.
Let $I = \{1,\ldots,k\}$ and $\Lambda = \{1,\ldots,l\}$.
Let $P = (p_{\lambda i})$.

We construct an automaton $\B$ with $\Sigma(\B) = S$.
Let $\B = (Q', B, \delta')$ with $B = A\cup C$, where $C = ((I\cup \{e\}) \times \Lambda)$
and $\delta'$ the transformation of $Q'\times B$ given by
\begin{align*}
\left( (j,x,\mu),(e,\lambda)\right) &\mapsto \left( (1,x,1), (j,\mu)\right)\\
\left( (j,x,\mu),(i,\lambda)\right) &\mapsto \left( (1,p_{\lambda j}x,1), (i,\mu) \right)\\
\left( (j,x,\mu), a\right)          &\mapsto \left( (1,x\pi_a,1), a\tau_x\right)
\end{align*}
for $x\in Q$, $i\in I$, $j\in I$, $\lambda,\mu\in \Lambda$ and $a\in A$,
and $\pi_b$, $\tau_x$ as in $\A$.

\begin{figure}[p]
\centering
\begin{tikzpicture}[x=30mm,y=30mm,every state/.style={rounded rectangle,inner sep=1mm,minimum size=1.75em}]
\node[state] (e) at (0,5) {$a$};
\node[state] (x) at (1,5) {$b$};
\node[state] (z) at (.5,4.5) {$0$};
\begin{scope}[every node/.style={font=\footnotesize}]
\draw[->] (x) edge[loop above] node[align=center] {$0|0$} (x);
\draw[->] (x) edge node[anchor=south] {$1|0$} (e);
\draw[->] (e) edge[loop above] node[align=center] {$0|0$\\$1|1$} (e);
\draw[->] (z) edge[loop below] node[align=center] {$0|0$\\$1|0$} (z);
\end{scope}
\draw[lightgray,thick,rounded corners] (-.5,4) rectangle (1.5,5.5);
\node[state] (1a1) at (0,3) {$(1,a,1)$};
\node[state] (1b1) at (1,3) {$(1,b,1)$};
\node[state] (1z1) at (.5,2) {$(1,0,1)$};
\node[state] (1a2) at (3,3) {$(1,a,2)$};
\node[state] (1b2) at (4,3) {$(1,b,2)$};
\node[state] (1z2) at (3.5,2) {$(1,0,2)$};
\node[state] (2a1) at (0,0) {$(2,a,1)$};
\node[state] (2b1) at (1,0) {$(2,b,1)$};
\node[state] (2z1) at (.5,1) {$(2,0,1)$};
\node[state] (2a2) at (3,0) {$(2,a,2)$};
\node[state] (2b2) at (4,0) {$(2,b,2)$};
\node[state] (2z2) at (3.5,1) {$(2,0,2)$};
\begin{scope}[every node/.style={font=\footnotesize}]
\draw[->] (1b1) edge[loop right] node[align=center] {$e\lambda|11$\\$i1|i1$\\$0|0$} (1b1);
\draw[->,rounded corners] (1b1) -- ($ (1b1) + (-.25,-.25) $) -- node[align=center,anchor=east,near start] {$i2|i1$} ($ (1z1) + (.25,.25) $) -- (1z1);
\draw[->] (1b1) edge node[anchor=north,align=center] {$1|0$} (1a1);
\draw[->] (1a1) edge[loop above] node[align=center] {$e\lambda|11$\\$i1|i1$\\$0|0$\\$1|1$} (1a1);
\draw[->,rounded corners] (1a1) -- ($ (1a1) + (.25,-.25) $) -- node[align=center,anchor=west] {$i2|i1$} ($ (1z1) + (-.25,.25) $) -- (1z1);
\draw[->] (1z1) edge[loop left] node[align=center,near start,anchor=north] {$e\lambda|11$\\$i\lambda|i1$\\$0|0$\\$1|0$} (1z1);
\draw[->,rounded corners] (1a2) -- ($ (1a2) + (-.2,.2)$) -- node[align=center,anchor=north,very near start] {$e\lambda|12$\,$i1|i2$\\$0|0$\,$1|1$} ($ (1a1) + (.2,.2)$) -- (1a1);
\draw[rounded corners] (1a2) -- node[align=center,anchor=south east,pos=.6,inner sep=.2mm] {$i2|i2$} ($ (1a2) + (-1,-1)$) -- (1z1);
\draw[->] (1z2) -- node[align=center,anchor=north,very near start] {$e\lambda|12$\,$i\lambda|i2$\\$0|0$\,$1|0$} (1z1);
\draw[line width=3pt,white] ($ (1b2) + (-.3,-.3)$) -- ($ (1b1) + (.3,-.3)$);
\draw[->,rounded corners] (1b2) -- ($ (1b2) + (-.3,-.3)$) -- node[align=center,anchor=south,pos=.1] {$e\lambda|12$\\$i1|i2$\\$0|0$} ($ (1b1) + (.3,-.3)$) -- (1b1);
\draw[->,rounded corners] (1b2) -- ($ (1b2) + (-.3,.3)$) -- node[anchor=north,very near start] {$1|0$} ($ (1a1) + (.3,.3) $) -- (1a1);
\draw[rounded corners] (1b2) -- ($ (1b2) + (0,-.2)$) -- ($ (1b2) + (-.6,-.8) $) -- node[anchor=south,very near start,align=center] {$i2|i2$} ($ (1a2) + (-.8,-.8) $) -- ($ (1a2) + (-.9,-.9) $);
\draw[rounded corners] (2b1) -- ($ (2b1) + (-.3,0) $) -- node[align=center,anchor=south west,inner sep=.2mm,near start] {$i1|i1$\\$1|0$} ($ (2a1) + (-.2,.9) $) -- ($ (1a1) + (-.2,-.2) $) -- (1a1);
\draw[rounded corners] (2b1) -- node[align=center,anchor=north west,inner sep=.2mm,near start] {$i2|i1$} ($ (2b1) + (.5,.5) $) -- ($ (1z1) + (1,-1) $) -- ($ (1z1) + (.5,-.5) $) -- (1z1);
\draw[->,rounded corners] (2a1) -- node[align=center,anchor=south west,inner sep=.1mm] {$e\lambda|21$\\$i1|i1$\\$0|0$\\$1|1$} ($ (2a1) + (-.3,.3) $) -- ($ (1a1) + (-.3,-.3) $) -- (1a1);
\draw[rounded corners] (2a1) -- ($ (2a1) + (.2,-.2) $) -- node[align=center,anchor=south west,near start,inner sep=.1mm] {$i2|i1$} ($ (2b1) + (.3,-.2) $) -- ($ (2b1) + (.5,0) $) -- ($ (2b1) + (.5,.75) $);
\draw[->] (2z1) -- node[align=center,anchor=west,pos=.3] {$e\lambda|21$\\$i\lambda|i1$\\$0|0$\\$1|0$} (1z1);
\draw[line width=3pt,white] ($ (1b2) + (0,-.4) $) -- ($ (1b1) + (.4,-.4) $);
\draw[rounded corners] (2b2) -- node[align=center,anchor=east,very near start] {$e\lambda|22$\\$i1|i2$\\$0|0$} ($ (1b2) + (0,-.4) $) -- ($ (1b1) + (.4,-.4) $) -- (1b1);
\draw[rounded corners] (2b2) -- ($ (2b2) + (.3,.3) $) -- node[anchor=east] {$1|0$} ($ (1b2) + (.3,.1)$) -- ($ (1b2) + (0,.4) $) -- ($ (1a1) + (.4,.4) $) -- (1a1);
\draw[rounded corners] (2b2) -- ($ (2b2) + (-.2,.2) $) -- node[anchor=south,align=center] {$i2|i2$} ($ (1z1) + (1.8,-1.8) $) -- ($ (1z1) + (1.6,-1.6) $);
\draw[rounded corners] (2a2) -- ($ (2a2) + (-.3,-.3) $) -- node[align=center,anchor=south,pos=.35] {$e\lambda|22$\\$i1|i2$\\$0|0$\\$1|1$} ($ (2a1) + (-.1,-.3) $) --  ($ (2a1) + (-.4,0) $) -- ($ (1a1) + (-.4,-.4) $) -- (1a1);
\draw[rounded corners] (2a2) -- ($ (1z1) + (2,-2) $) -- node[align=center,anchor=south west,inner sep=-.2mm,near start] {$i2|i2$} ($ (1z1) + (.5,-.5) $);
\draw[->,rounded corners] (2z2) -- node[align=center,anchor=south,very near start] {$e\lambda|22$\,$i\lambda|i2$\\$0|0$\,$1|0$} ($ (1z1) + (1,-1) $) -- (1z1);
\draw[line width=3pt,white] ($ (2b1) + (0,.5) $) -- ($ (1b1) + (0,-.5) $);
\draw[->] (2b1) edge node[align=center,anchor=west,near start] {$e\lambda|21$\\$i1|i1$\\$0|0$} (1b1);
\end{scope}
\end{tikzpicture}
\caption[]{Above, the automaton for semigroup $F^0$, where $F$ is the free monoid generated by $b$, and with identity
  $a$. Below, the automaton for $\mathcal{M}[I,F^0,\Lambda,P]$, where $I = \Lambda = \{1,2\}$ and $P = \bigl[\begin{smallmatrix}a & a\\ 0 & 0\end{smallmatrix}\bigr]$.}
\label{fig:reesmatrix}
\end{figure}

This construction somewhat resembles the automaton for the wreath product,
in that it is designed to perform the appropriate `twist' to the action
of $M$.
The `ideal' strings, which we use to distinguish elements of $S$, are those in
$D = (\{e\}\times \Lambda)A^*$.  For $(j_i,x_i,\mu_i)\in (Q')^*$,
let $(j,x,\mu)$ be the element of $S$ represented by
$(j_1,x_1,\mu_1)\ldots (j_k,x_k,\mu_k)$.
For $(e,\lambda)\alpha\in D$ we have
\begin{align*}
& (e,\lambda)\alpha\cdot (j_1,x_1,\mu_1)\ldots (j_k,x_k,\mu_k) \\
={}& \left[(j_1,\mu_1) (\alpha\cdot x_1)\right]\cdot (j_2,x_2,\mu_2)\ldots (j_k,x_k,\mu_k)\\
={}& \left[(j_1,\mu_2) (\alpha\cdot x_1 p_{\mu_1 \lambda_2} x_2\right] \cdot (j_3,x_3,\mu_3)\ldots (j_k,x_k,\mu_k)\\
={}& (j_1,\mu_k) (\alpha\cdot x_1 p_{\mu_1 \lambda_2} x_2 \ldots p_{\mu_{k-1}\lambda_k} x_k)\\
={}& (j,\mu) (\alpha\cdot x).
\end{align*}
Thus $\B$ defines a faithful action of $S$ on $D$.

It is then easy to see that this action extends to an action on the whole of $B^*$.
We can write any string in $B^*$ as an alternating product of strings in
$A^*$ and $C^*$.
Let $w\in (Q')^*$ with $w=_S (j,x,\mu)\in Q'$.
When $w$ acts on $\alpha\in A^*$, the output string is
$\alpha\cdot x$ and the automaton ends in state $(1,x|_\alpha,1)$ , while
when $w$ acts on $(i_1,\lambda_1)\ldots,(i_k,\lambda_k)\cdot (j,x,\mu)$,
the output string is $(i_1\cdot j, \mu) (i_2,1)\ldots (i_k\cdot j, 1)$
(where $I\cup \{e\}$ is treated as a left zero semigroup with adjoined identity $e$)
and the automaton ends in state $p_{\lambda_1 j} p_{\lambda_2 1} \ldots p_{\lambda_k 1} x$.
Since these end states and outputs depend only on $(j,x,\mu)$ and the input string,
we conclude that $\A$ defines an action of $S$ on alternating products of
strings in $A^*$ and $C^*$, that is, on $B^*$.  Moreover, this action is faithful,
since elements can be distinguished by their actions on $D$.  Hence $\Sigma(\A)\cong S$.
\end{proof}

\section{Strong semilattices of semigroups}

We recall the definition of a strong semilattice of semigroups here, and refer the reader to
\cite[Sect.~4.1]{howie_fundamentals} for further background reading:

\begin{definition}
  Let $Y$ be a semilattice. Recall that the meet of $\alpha,\beta \in Y$ is denoted $\alpha \land \beta$. For each $\alpha \in Y$, let
  $S_\alpha$ be a semigroup. For $\alpha \geq \beta$, let $\phi_{\alpha,\beta} : S_\alpha \to S_\beta$ be a homomorphism
  such that
\begin{enumerate}
\item For each $\alpha \in Y$, the homomorphism $\phi_{\alpha,\alpha}$ is the identity mapping.
\item For all $\alpha,\beta,\gamma \in Y$ with $\alpha \geq \beta \geq \gamma$,
\[
\phi_{\alpha,\beta}\phi_{\beta,\gamma} = \phi_{\alpha,\gamma}.
\]
\end{enumerate}
The \defterm{strong semilattice of semigroups} $S =
\mathcal{S}[Y;S_\alpha;\phi_{\alpha,\beta}]$ consists of the disjoint
union $\bigcup_{\alpha \in Y} S_\alpha$ with the following
multiplication: if $x \in S_\alpha$ and $y \in S_\beta$, then
\[
xy = (x\phi_{\alpha,\alpha \land \beta})(y\phi_{\beta,\alpha\land\beta}),
\]
where $\alpha \land \beta$ denotes the greatest lower bound of
$\alpha$ and $\beta$.
\end{definition}

The following result proves that a certain type of strong semilattice of automaton semigroups is itself an automaton
semigroup. Although it is of restricted scope, this result is of independent interest, and it will also be applied in
the following section on small extensions of automaton semigroups.

\begin{proposition}\label{prop:semilattice}
  Let $S_1,\ldots,S_k$ be automaton semigroups and $T$ a finite semigroup with a right zero. Let $Y$ be the semilattice
  where all the $S_i$ are mutually incomparable and all greater than $T$, which is the minimum element of $Y$. Then the
  strong semilattice $S = {\cal S}[Y; S_1,\ldots,S_k,T; \phi_1,\ldots,\phi_k]$ with $\phi_i:S_i\rightarrow T$ is an
  automaton semigroup.
\end{proposition}

\begin{proof}
Let $z$ be a distinguished right zero in $T$.
For $1\leq i\leq k$, let $\A_i = (Q_i, A_i, \delta_i$) be an automaton
for $S_i$.
Let $P$ be a copy of $T$ and $B$ a copy of $T^1$ and for $1\leq i\leq k$ let
$A_i^0 = A_i\cup \{0\}$, where $0$ is a new symbol not in any $A_i$ or $B$.
Let $\B = (Q,C,\delta)$, where $Q = Q_1\cup \ldots\cup Q_k\cup P$,
$C = A_1^0\times \ldots \times A_k^0 \times B$ and $\delta$ is defined as follows:
Let $c = (a_1,\ldots,a_k,b)$ be any element of $C$.
For $p\in P$ we define $(p,c)\delta = (z, (0,\ldots,0,bp))$.
For $a\in A_i$ and $q\in Q_i$, let $\pi_a$ and $\tau_q$ be as in $\A_i$.
We extend $\tau_q$ to $\tau'_q:A_i^0\rightarrow A_i^0$ by defining $a\tau'_q$ to be
$a\tau_q$ if $a\neq 0$ and $0$ if $a=0$.  The output on reading $c$
in state $q$ is $(a'_1,\ldots,a'_k,b\phi_i(q))$, where $a'_i = a_i\tau'_q$
and $a'_j = 0$ for all $j\neq i$; and we move to state $q\pi_{a_i}$ if $a_i\neq 0$
and to state $z$ otherwise.

\begin{figure}[p]
\centering
\begin{tikzpicture}[x=30mm,y=30mm,every state/.style={rounded rectangle,inner sep=1mm,minimum size=1.75em}]
\node[state] (a1) at (0,3.1) {$a$};
\node[state] (b1) at (1,3.1) {$b$};
\node[state] (c1) at (2,3.1) {$c$};
\node[state] (d1) at (3,3.1) {$d$};
\node[state] (e1) at (1,2) {$e$};
\node[state] (z1) at (2,2) {$z$};
\draw[->] (a1) edge[loop above] node {$1|1$} (a1);
\draw[->] (a1) edge[bend left=10] node[anchor=south] {$2|1$} (b1);
\draw[->] (b1) edge[loop above] node {$2|2$} (b1);
\draw[->] (b1) edge[bend left=10] node[anchor=north] {$1|2$} (a1);
\draw[->] (c1) edge[loop above] node {$\beta|\beta$} (c1);
\draw[->] (d1) edge[loop above] node {$1|1$} (d1);
\draw[->] (d1) edge node[anchor=south] {$2|1$} (c1);
\draw[->] (e1) edge[loop above] node[align=center] {$e|e$\\$z|z$} (e1);
\draw[->] (z1) edge[loop above] node[align=center] {$e|z$\\$z|z$} (z1);
\draw[lightgray,thick,rounded corners] (-.45,2.8) rectangle (1.45,3.6);
\draw[lightgray,thick,rounded corners] (1.55,2.8) rectangle (3.45,3.6);
\draw[lightgray,thick,rounded corners] (.5,1.7) rectangle (2.5,2.7);
\node[state] (a) at (0,1) {$a$};
\node[state] (b) at (1,1) {$b$};
\node[state] (c) at (2,1) {$c$};
\node[state] (d) at (3,1) {$d$};
\node[state] (e) at (1,0) {$e$};
\node[state] (z) at (2,0) {$z$};
\begin{scope}[every node/.style={font=\footnotesize}]
\draw[->] (a) edge[loop above] node[align=center] {$1\beta x|10x$\\$10x|10x$} (a);
\draw[->] (a) edge[bend left=10] node[align=center,anchor=south] {$2\beta x|10x$\\$20x|10x$} (b);
\draw[->] (a) edge[bend right=10] node[align=center,anchor=north east,inner sep=.2mm,pos=.4] {$0\beta x|00x$\\$00x|00x$} (z);
\draw[->] (b) edge[loop above] node[align=center] {$2\beta x|20z$\\$20x|20z$} (b);
\draw[->] (b) edge[bend left=10] node[align=center,anchor=north,near start] {$1\beta x|20z$\\$10x|20z$} (a);
\draw[->] (b) edge node[align=center,anchor=north east,inner sep=.2mm,pos=.4] {$0\beta x|00z$\\$00x|00z$} (z);
\draw[->] (c) edge[loop above] node[align=center] {$\alpha\beta x|0\beta x$\\$0\beta x|0\beta x$} (c);
\draw[->] (c) edge node[align=center,anchor=west,near start] {$\alpha 0x|00x$\\$00x|00x$} (z);
\draw[->] (d) edge[loop above] node[align=center] {$\alpha 1x|01z$\\$01x|01z$} (d);
\draw[->] (d) edge node[align=center,anchor=south] {$\alpha 2x|01z$\\$02x|01z$} (c);
\draw[->] (d) edge[bend left=10] node[align=center,anchor=north west,inner sep=.2mm] {$\alpha 0x|00z$\\$00x|00z$} (z);
\draw[->] (e) edge node[align=center,anchor=north,near start] {$\alpha\beta x|00x$\\$\alpha 0x|00x$\\$0\beta x|00x$\\$00x|00x$} (z);
\draw[->] (z) edge[loop below] node[align=center] {$\alpha\beta x|00z$\\$\alpha 0x|00z$\\$0\beta z|00z$\\$00x|00z$} (z);
\end{scope}
\draw[lightgray,thick,rounded corners] (-.5,-.75) rectangle (3.5,1.6);
\end{tikzpicture}
\caption{Below, the automaton for the strong semilattice of semigroups $\mathcal{S}[Y;S_1,S_2,T;\phi_i]$,
  where $S_1
  = F_2$ (with basis $\{a,b\}$), $S_2 = \N_0$ (with $c$ being the additive identity of $\N_0$ and
  $d$
  representing the natural number $1$),
  and $T
  = \{e,z\}$ (where $e^2 = e$ and $ez = ze = z^2 = z$), where $a\phi_1 = e$ and $b\phi_1 = z$, and $c\phi_2 =
  e$ and $d\phi_1 = z$. Throughout the diagram, $\alpha$ and $\beta$ are arbitrary symbols in $\{1,2\}$ and
  $x$
  is an arbitrary symbol in $\{e,z\}$.
  For reasons of space, triples $(\kappa,\lambda,\mu)$
  are abbreviated $\kappa\lambda\mu$. Above, the original automata for (clockwise from top left) $F_2$, $\N_0$, and $T$.}
\label{fig:strongsemilattice}
\end{figure}
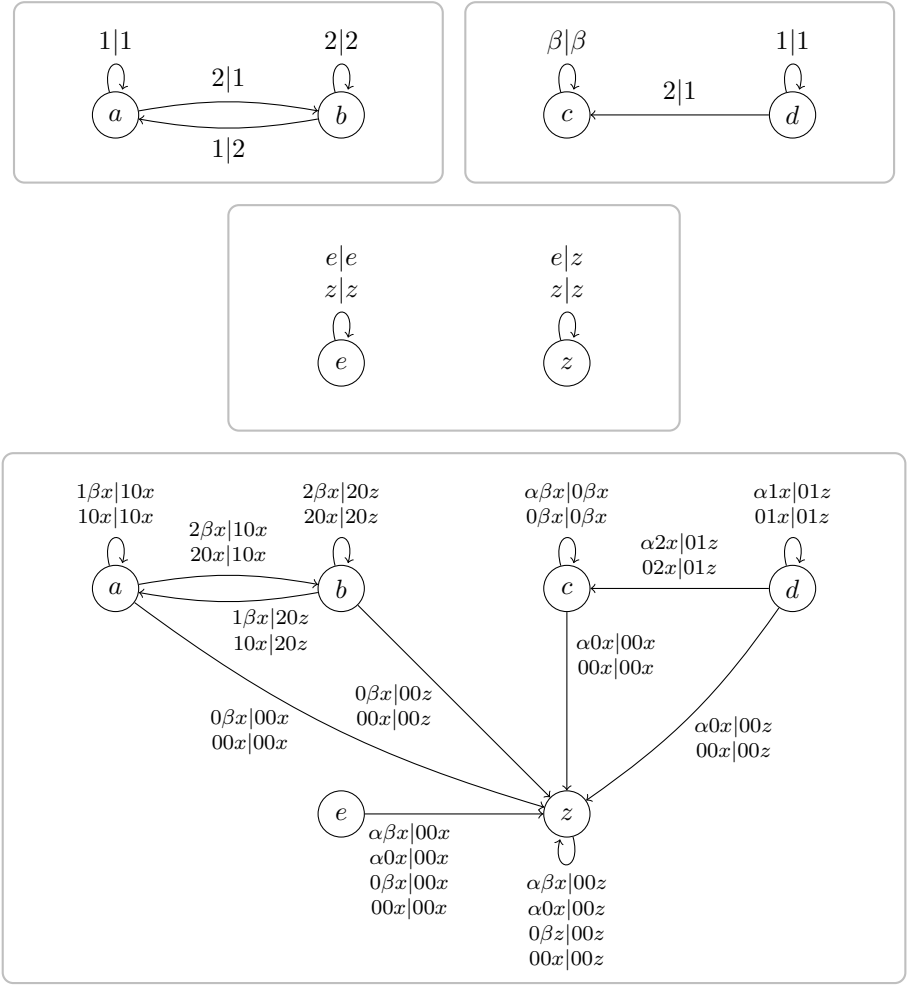

For $p\in P$ we have $C^\omega\cdot p\subseteq C (0,\ldots,0,z)^\omega$, so the action of
$P^+$ on $C^\omega$ depends only on its action on $C$, which is essentially
just the action by right multiplication on $B=T^1$, so $\langle P\rangle\cong T$.

Fix some $i\in \{1,\ldots,k\}$, and let $C_i$ be the subset of $C$ consisting of
all tuples without $0$ in their $i$-th component.  Then every string in $C^*$ is
a prefix of some $\alpha c \beta$, where $\alpha\in C_i^*$, $c\in C\setminus C_i$
and $\beta\in C^\omega$.  Now for $w\in Q_i^+$
\[ \alpha c \beta \cdot w = (\alpha\cdot w) (c\cdot w|_\alpha) z^\omega.\]
Since the state transitions during the computation of $\alpha\cdot w$ are
governed entirely by the $i$-th components of the symbols in $\alpha$, which
in turn are the same as in $\A_i$, $w|_\alpha$ depends only on $\alpha$ and
the element $s_w$ of $S_i$ represented by $w$.  For a string $\gamma\in C^*$
and $j\in \{1,\ldots,k+1\}$, denote by $\gamma(j)$ the string obtained by reading off the
$j$-th components of symbols in $\gamma$. Then we have
$(\alpha\cdot w)(i) = \alpha(i)\cdot w$ and $(\alpha\cdot w)(j) = 0^m$
for $j\notin \{i,k+1\}$, where $m = |\alpha|$.
Thus the action of $Q_i^+$ on the first $k$ components of strings in $C^*$ is an action
of $S_i$.  Now we need to check the action on the final component.
Let $\alpha_j$ be the prefix of $\alpha$ of length $j$, for $0\leq j\leq m$,
and let $b_j$ be the final component of the $j$-th symbol in $\alpha$.
Then $(\alpha\cdot w)(k+1) = b'_1\ldots b'_m$ with $b'_j =  b_j\phi_i(w|_{\alpha_{j-1}})$.
Since each $w|_{\alpha_j}$ depends only on $\alpha$ and $s_w$, this concludes the proof that
the action of $Q_i^+$ on $C^*$ is an action of $S_i$.
Moreover, this action is faithful, due to the action on the $i$-th component of strings
being identical to the action in $\A_i$.  So $Q_i$ generates a subsemigroup of $\Sigma(\B)$
isomorphic to $S_i$.

Finally, we establish that the multiplication works correctly outside of the
defining subsemigroups $S_1,\ldots,S_k,T$.
For any word $w\in Q^+$ containing elements from more than one defining subsemigroup,
we have $C^\omega \cdot w\subseteq C(0,\ldots,0,z)^\omega$.
This is because acting by a state in $P$, or acting on a tuple with $0$ in the $i$-th
position by a state in $Q_i$, both cause a transition to the state $z$, which sends all
strings to $(0,\ldots,0,z)$.
Hence the action of `multi-subsemigroup' words on $C^\omega$ is determined entirely
by their action on $C$.  For $c = (a_1,\ldots,a_k,b)\in C$, $p\in P$, $q_i\in Q_i$
and $q_j\in Q_j$ with $i\neq j$, we have
\begin{align*}
c\cdot pq_i &= (0,\ldots,0,bp)\cdot q_i = (0,\ldots,0,bp\phi_i(q)) = c\cdot \ol{pq_i},\\
c\cdot q_ip &= (0,\ldots,0,a_i\tau_{q_i},0,\ldots,0,b\phi_i(q_i))\cdot p = (0,\ldots,0,b\phi_i(q)p) = c\cdot \ol{q_ip},\\
c\cdot q_i q_j &= (0,\ldots,0,a_i\tau_{q_i},0,\ldots,0,b\phi_i(q_i))\cdot q_j = (0,\ldots,0,b\phi_i(q_i)\phi_j(q_j)) = c\cdot \ol{q_iq_j},
\end{align*}
where $\ol{w}$ denotes the element of $S$ represented by $w$.
Hence $\Sigma(\B)\cong S$.
\end{proof}

\section{Small extensions}
\label{sec:smallext}

Recall that if $S$ is a semigroup and $T$ is a subsemigroup of $S$ with $S \setminus T$ finite, then $S$ is
a \defterm{small extension} of $T$.

In this section, we present some examples of small extensions of automaton semigroups
that are again automaton semigroups.  Our first example is the $k=1$ case of Proposition~\ref{prop:semilattice}.
(For $k\geq 2$, the semigroups in Proposition~\ref{prop:semilattice} are not small extensions of
automaton semigroups.)

\begin{example}\label{ex:semilattice}
  Let $S$ be an automaton semigroup and let $T$ be a finite semigroup with a right zero.  Then if $\phi: S\rightarrow T$
  is any homomorphism, the strong semilattice $\S(\{1<2\};S,T;\phi)$ is an automaton semigroup.
\end{example}

This example is interesting because it has the potential to lead to an answer to one of
the basic open questions about automaton semigroups (see \cite[Open problem~5.3]{c_1auto}):
\begin{question}\label{q:zero}
Does there exist a non-automaton semigroup $S$ such that $S^0$ is an automaton semigroup?
\end{question}
If we can prove that some similar strong semilattice with the finite
semigroup $T$ \emph{not} having a right zero is \emph{not} an automaton semigroup, then we will have an example of a
semigroup that is not an automaton semigroup, but becomes one on adjoining a zero.
We conjecture that the following strong semilattices of semigroups are not automaton semigroups:
\begin{itemize}
\item $F_2=\langle x,y\rangle$ above $C_2 = \{e,f\}$ with $\phi(x)=e, \phi(y)=f$.
\item $\N_0=\langle 0,1\rangle$ above $C_2 = \{e,f\}$ with $\phi(0)=e, \phi(1)=f$.
\end{itemize}

Example~\ref{ex:semilattice} can be generalised in a different direction:

\begin{proposition}\label{prop:smallext}
Let $S_1$ be an automaton semigroup and $S_2$ a finite semigroup with a right zero.
Then any semigroup $S = S_1\cup S_2$ having $S_2$ as an ideal is an automaton semigroup.
\end{proposition}
\begin{proof}
Let $\A_1 = (Q_1,A,\delta_1)$ be any automaton for $S_1$
and let $\A_2 = (Q_2,B,\delta_2)$
be the automaton for $S_2$ having $Q_2 = S_2$, with $B$ be a copy of $S_2^1$ (whose elements
we will denote in the form $\ol{b}$) and
$(t,\ol{b})\delta_2 = (z,\ol{bt})$ for $t\in Q_2$, $\ol{b}\in B$, where $z$ is some right zero in $Q_2$.
We construct an automaton $\A$ with $\Sigma(\A) = S$ as follows:
For each $s\in S_1$, let $\lambda_s$ and $\rho_s$ be the transformations of $B$ induced by the left
and right actions respectively of $s$ on $S_2$.  Define $\Lambda = \{\lambda_s\mid s\in S_1\}\cup \{\iota_B\}$.
Since $\Lambda$ is a subsemigroup of the full (left) transformation semigroup of $B$, it is finite.
Let $C = (A\times \Lambda)\cup B$ and let $\A = (Q_1\cup Q_2, C, \delta)$ with $\delta$ given by
\begin{align*}
(x, (a,\mu)) &\mapsto (x\pi_a, (a\tau_x, \mu\lambda_x)) &  (x, \ol{b}) &\mapsto (z, \ol{b\rho_x})\\
(y, (a,\mu)) &\mapsto (z,\ol{\mu y})           &  (y, \ol{b}) &\mapsto (z, \ol{by})
\end{align*}
for $x\in Q_1$, $y\in Q_2$, $a\in A$, $\mu\in \Lambda$, $\ol{b}\in B$ and for $\pi_a$ and $\tau_x$
as in $\A_1$.

We begin by considering the action of words in $Q_2^+$.
For $(a,\mu)\in A\times \Lambda$, $\ol{b}\in B$, $\alpha\in C^\omega$ and $y\in Q_2$,
we have $(a,\mu)\alpha\cdot y = \ol{\mu y} z^\omega$ and $\ol{b}\alpha\cdot y = \ol{by}z^\omega$.
Thus the action of $w\in Q_2^+$ on $C^\omega$ depends only on its actions on $\Lambda$ and on $B$,
both of which depend only on the element of $S_2$ represented by $w$.  Moreover, the action on
$\Lambda$ is faithful, since $\iota_B \ol{b} = \ol{b}$ for all $\ol{b}\in B$.
Hence the subsemigroup of $\Sigma(\A)$ generated by $Q_2$ is isomorphic to $S_2$.

Next we consider the action of words in $Q_1^+$.
For $\alpha\in (A\times \Lambda)^*$, $\ol{b}\in B$, $\gamma\in C^\omega$ and $w\in Q_1^+$ we have
\[ \alpha \ol{b} \gamma \cdot w = (\alpha\cdot w) \ol{b\rho_{w|_\alpha}} z^\omega.\]
Let $s_w$ be the element of $S_1$ represented by $w$ in $\Sigma(\A_1)$.
If $\alpha = (a_1,\mu_1)\ldots (a_k,\mu_k)$, let $\alpha\cdot w = (c_1,\nu_1)\ldots (c_k,\nu_k)$.
Then in $\A_1$ we have $a_1\ldots a_k\cdot w = c_1\ldots c_k$.
Since $\Sigma(\A_1) = S_1$, this means that $\A$ restricted to states in $Q_1$ defines a faithful
action of $S_1$ on the first component of strings in $(A\times \Lambda)^*$.
Also note that $\rho_{w|_\alpha}$ depends only on $\alpha$ and $s_w$,
so that the only possible obstacle to $\langle Q_1\rangle$ being isomorphic to $S_1$ would be the action
of $w$ on the second component of strings in $(A\times \Lambda)^*$ not depending only on $s_w$.
However, if for $0\leq i\leq k-1$ we let $\alpha_i$ be the prefix of $\alpha$ of length $i$, then
$\nu_{i+1} = \mu_{i+1} \lambda_{w|_{\alpha_i}}$, and so the second component of $\alpha\cdot w$
also depends only on $s_w$ and $\alpha$.  Hence $Q_1$ generates a subsemigroup of $\Sigma(\A)$
isomorphic to $S_1$.

It remains to establish that products $s_1 s_2$ and $s_2 s_1$ with $s_i\in S_i$ act correctly.
Since $C^\omega\cdot S_2\subseteq C z^\omega$, we only need to consider the action on $C$.
For $a\in A$, $\mu\in \Lambda$, $\ol{b}\in B$ and $s_i\in S_i$ with $s_1s_2 =_S s$ and $s_2s_1 =_S s'$  we have
\begin{align*}
(a,\mu)\cdot s_1 s_2 &= (a\tau_{s_1}, \mu\lambda_{s_1})\cdot s_2 = \ol{\mu\lambda_{s_1} s_2} = (a,\mu)\cdot s,\\
(a,\mu)\cdot s_2 s_1 &= \ol{\mu s_2}\cdot s_1 = \ol{\mu s_2 \rho_{s_1}} = (a,\mu)\cdot s',\\
\ol{b}\cdot s_1 s_2 &= \ol{b\rho_{s_1}}\cdot s_2 = \ol{b\rho_{s_1} s_2} = \ol{b}\cdot s,\\
\ol{b}\cdot s_2 s_1 &= \ol{b s_2}\cdot s_1 = \ol{bs_2\rho_{s_1}} = \ol{b}\cdot s'.
\end{align*}
Hence $\Sigma(\A)\cong S$ and so $S$ is an automaton semigroup.
\end{proof}

In \cite{maltcev_hopfian}, Maltcev and Ru\v{s}kuc gave a construction for a type of small extension
as follows.  Let $S$ be any semigroup acting on a finite set $X$ (on the right).
Let $x^s$ denote the result of acting on $x\in X$ by $s\in S$.
The semigroup $S[X]$ is defined to be the disjoint union of $S$ and $X$, with multiplication given by
\[
st = st,\qquad xs = x^s, \qquad sx = x, \qquad xy = y,
\]
for $s,t\in S$, $x,y\in X$.

\begin{example}
Let $S$ be an automaton semigroup and $X$ a finite set.  Then the semigroup $S[X]$
is an automaton semigroup.
\end{example}
\begin{proof}
This follows immediately from Proposition~\ref{prop:smallext}, since $X$ is
both an ideal of $S[X]$ and a right zero semigroup.
Note that the automaton from Proposition~\ref{prop:smallext} may be simplified
in this case, since every element of $S$ acts as a left identity on $X$ and hence
the set $\Lambda$ is superfluous.  An example is shown in Figure~\ref{fig:smallext}.
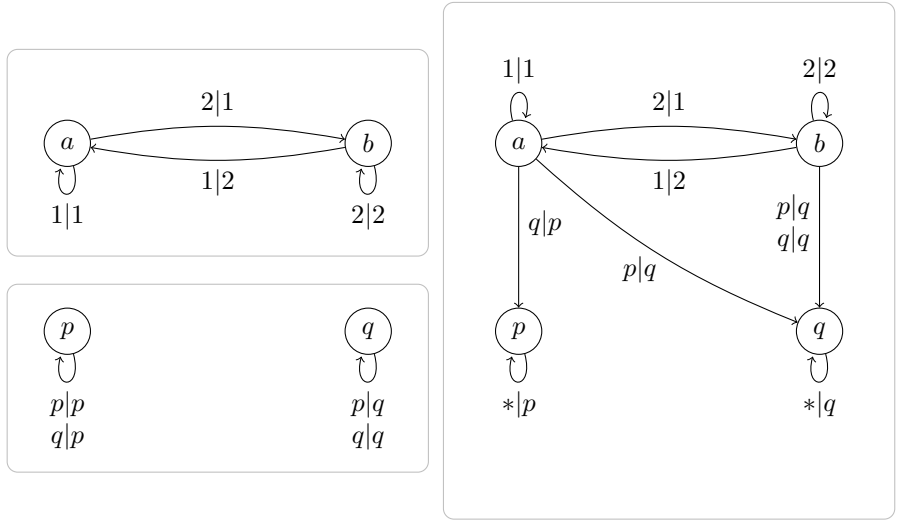
\begin{figure}[t]
\centering
\begin{tikzpicture}[x=20mm,y=25mm,every state/.style={rounded rectangle,inner sep=1mm,minimum size=1.75em}]
\node[state] (x1) at (0,1) {$a$};
\node[state] (y1) at (2,1) {$b$};
\node[state] (e1) at (0,0) {$p$};
\node[state] (z1) at (2,0) {$q$};
\draw[->] (x1) edge[loop below] node {$1|1$} (x1);
\draw[->] (x1) edge[bend left=10] node[anchor=south] {$2|1$} (y1);
\draw[->] (y1) edge[loop below] node {$2|2$} (y1);
\draw[->] (y1) edge[bend left=10] node[anchor=north] {$1|2$} (x1);
\draw[->] (e1) edge[loop below] node[align=center] {$p|p$\\$q|p$} (e1);
\draw[->] (z1) edge[loop below] node[align=center] {$p|q$\\$q|q$} (z1);

\draw[lightgray,rounded corners] (-.4,-.75) rectangle (2.4,.25);
\draw[lightgray,rounded corners] (-.4,.4) rectangle (2.4,1.5);
\node[state] (a) at (3,1) {$a$};
\node[state] (b) at (5,1) {$b$};
\node[state] (e) at (3,0) {$p$};
\node[state] (z) at (5,0) {$q$};
\draw[->] (a) edge[loop above] node[align=center] {$1|1$} (a);
\draw[->] (a) edge[bend left=10] node[align=center,anchor=south] {$2|1$} (b);
\draw[->] (b) edge[loop above] node[align=center] {$2|2$} (b);
\draw[->] (b) edge[bend left=10] node[align=center,anchor=north] {$1|2$} (a);
\draw[->] (e) edge[loop below] node[align=center] {$*|p$} (e);
\draw[->] (z) edge[loop below] node[align=center] {$*|q$} (z);
\draw[->] (a) edge node[align=center,anchor=west,pos=.4] {$q|p$} (e);
\draw[->] (a) edge[bend right=10] node[align=center,anchor=north east,inner sep=.2mm] {$p|q$} (z);
\draw[->] (b) edge node[align=center,anchor=east,pos=.4] {$p|q$\\$q|q$} (z);
\draw[lightgray,rounded corners] (2.5,-1) rectangle (5.5,1.75);
\end{tikzpicture}
\caption{On the right, the automaton for the semigroup $F_2[X]$ with $F_2 = \langle a,b\rangle$,
$X=\{p,q\}$ and $p^a=q$, $q^a=p$, $p^b=q$, $q^b=q$. On the left, the original automata for
  $F_2$ and $X$.}
\label{fig:smallext}
\end{figure}
\end{proof}

It is not obvious how to construct a (right) automaton for the dual of $S[X]$
(where $S$ acts on the left and $X$ is a left zero semigroup),
because the idea of Proposition~\ref{prop:smallext} relies heavily on using a right zero
to `forget' information once it is no longer required.  A left zero cannot
be used in the same way.
This leads to the following question.

\begin{question}
Does there exist a semigroup $S$ such that $S$ is a right automaton semigroup
but not a left automaton semigroup?
\end{question}

\section{Some new examples of non-automaton semigroups}
\label{sec:examples}

There is so far no general method known for proving a finitely generated semigroup not to be an automaton semigroup,
other than by showing that it fails to have one of the properties satisfied by all automaton semigroups, such as being
residually finite or having solvable word problem.  Previously, the only known example of a semigroup that has these
`general' automaton semigroup properties but that is known \emph{not} to arise as an automaton semigroup was $\N$, the
free semigroup of rank $1$ \cite[Proposition~4.3]{c_1auto}.

Every subsemigroup of $\N$ is finitely generated \cite[Theorem~2.7]{rosales_numerical}. (This also follows from noting
that any subsemigroup $S$ of $\N$ is isomorphic to one whose elements have least common multiple $1$, and that such a
semigroup contains all but finitely many natural numbers by Euclid's algorithm and is thus a large subsemigroup of the
finitely generated semigroup $\N$. Finite generation is preserved on passing to large subsemigroups
\cite[Theorem~1.1]{ruskuc_largesubsemigroups}.) Furthermore, subsemigroups of $\N$ are residually finite.  In this
section we show that they do not arise as automaton semigroups.  The importance of this result is that we now have a
countable set of semigroups that satisfy the usual `general' properties of automaton semigroups, but that are not
actually automaton semigroups.

In fact we will show slightly more.  We establish that no subsemigroup of $\N^0$
(the free semigroup of rank $1$ with a zero adjoined, not to be confused with
$\N_0$, the free monoid of rank $1$) is an automaton semigroup.
This establishes that subsemigroups of $\N$ are not potential examples for a `yes'
answer to Question~\ref{q:zero} on the existence of non-automaton semigroups which
become automaton semigroups upon adjoining a zero.

Our approach is akin to the proof of \cite[Proposition~4.3]{c_1auto},
in that we use wreath recursions to show that any hypothetical automaton for
a subsemigroup of $\N$ must contain a state corresponding to a periodic element.
We make use of the following simple lemma.

\begin{lemma}
\label{lem:periodic}
Let $\A = (Q,B,\delta)$ be an automaton such that $\Sigma(\A)$ has a zero.
If there exist $q, z\in Q$, with $z$ representing the
zero element of $S$, such that $q$ recurses only to itself and $z$,
then $q$ represents a periodic element of $S$.
\end{lemma}
\begin{proof}
Let $q = (q_1,\ldots,q_k)\tau$ be the wreath recursion for $q$,
with $q_i\in \{q,z\}$.  Then for any $n$, we have
$q^n = (u_{n1},\ldots,u_{nm})\tau^n$,
where each $u_{ni}$ can be expressed as a product of $n$ elements from
$\{q,z\}$, and is hence in $\{q^n, z\}$.
But this means that two distinct powers of $q$ must have identical recursion
patterns
(that is, there exist distinct $m,n$ such that $\tau^m = \tau^n$, with
$u_{mi} = q^m$ if and only if $u_{ni} = q^n$)
and hence represent the same element of $S$, and so $q$ is periodic.
\end{proof}

\begin{theorem}
\label{thm:nosubsemigroupofn}
No subsemigroup of $\N^0$ except the trivial subsemigroup $\{0\}$ is an automaton semigroup.
\end{theorem}
\begin{proof}
  Let $S$ be a subsemigroup of $\N^0$ that is not the trivial subsemigroup.  Let $z$ denote the zero of $\N^0$; note
  that $z$ may not be an element of $S$. Suppose that $S = \Sigma(\A)$ for some automaton $\A$, and let $Q$ be the state
  set of $\A$.  We denote a state in $Q$ representing an element $i$ of $\N\cup \{z\}$ by $q_i$. Note that since $S$ is
  non-trivial, there exists elements $q_i \in Q$ with $i \neq z$. Let $k\in \N$ be maximal such that $q_k\in Q$ and let
  $l\in \N$ be minimal such that $q_l\in Q$. Since $q_k,q_l \in Q$, the wreath recursions for $q_k$ and $q_l$ are given
  by $q_k = (u_1,\ldots,u_d)\rho$ and $q_l = (v_1,\ldots,v_d)\sigma$ for some $u_i,v_i\in\{z,q_l,q_{l+1},\ldots, q_k\}$.
  If $k=l$, then this means that $q_k$ recurses only to $z$ and itself and is thus periodic by Lemma~\ref{lem:periodic},
  which contradicts the fact $S$ has no periodic elements except possibly $0$.

  So assume $l<k$.  Apply the
  formula $\eqref{eq:wreath}$ for multiplying wreath recursions to compute a wreath recursion for $q_k^l$ from the wreath recursion for $q_k$. Similarly, compute a wreath recursion for $q_l^k$ from the wreath recursion for $q_l$.  This gives two wreath recursions
\begin{align*}
q_k^l &= (w_1,\ldots,w_d)\rho^l\\
q_l^k &= (x_1,\ldots,x_d)\sigma^k
\end{align*}
where each $w_i$ is in $Q^l$, while each $x_i$ is in $Q^k$, and $w_i=_S x_i$ for all $i$ since
$q_k^l =_S kl = lk =_S q_l^k$.  If $w_i =_S z$ for all $i$, then $q_{kl}$ is periodic, so assume that some
$w_i\neq_S z$.  Clearly one way to achieve $w_i=_S x_i$ is if $w_i=q_k^l$ and $x_i=q_l^k$.  And indeed this is the only
possible solution, since $w_i$ is a product of length $l$ and any length-$l$ product of states other than $q_k^l$ itself
represents a smaller element of $S$ than $q_k^l$, while $x_i$ is a product of length $k$, and any length-$k$ product of
states other than $q_l^k$ itself represents a larger element of $S$ than $q_l^k$ (since the states are all in
$\{q_l,\ldots,q_k\}$). Hence if $w_i =_S x_i \neq_S z$, then $w_i = q_k^l$, and $x_i = q_l^k$. This means that the
wreath recursions for $q_k^l$ and $q_l^k$ recurse only to themselves and $z$ and are thus periodic, again by
Lemma~\ref{lem:periodic}.  But this is a contradiction, since $S$ has no periodic elements except possibly $0$, and
hence the automaton $\A$ cannot exist.
\end{proof}

\begin{question}
  Is there some general technique -- perhaps a kind of `pumping lemma' -- that gives a general tool for proving a
  finitely generated residually finite semigroup is not an automaton semigroup?
\end{question}

\section{Acknowledgements}

We thank the anonymous referee and Jan Philipp W\"{a}chter for pointing out errors in previous versions of this paper.


\end{document}